\documentclass{birkau}
\usepackage{graphicx}
\usepackage{newlattice}

\DeclareMathOperator{\Princ}{Princ}
\DeclareMathOperator{\Fuse}{Fuse}
\DeclareMathOperator{\Split}{Split}
\DeclareMathOperator{\Min}{Min}

\DeclareMathOperator{\Base}{Base}

\DeclareMathOperator{\Frame}{Frame}
\DeclareMathOperator{\Framew}{Frame_W}
\DeclareMathOperator{\Bridge}{Bridge}
\newcommand{\Fused}{\gi}

\newcommand{\gsum}{\overset{\text{\large \tbf .}}{+}} 

\theoremstyle{plain}
\newtheorem{theorem}{Theorem}

\newtheorem{lemma}[theorem]{Lemma}
\newtheorem{illustration}[theorem]{Illustration}
\newtheorem{corollary}[theorem]{Corollary}

\theoremstyle{definition}
\newtheorem{definition}[theorem]{Definition}

\begin{document}
\title[Minimal representations by principal congruences]{Minimal representations of a finite distributive lattice\\
by principal congruences of a lattice}  
\author{G. Gr\"{a}tzer} 
\email[G. Gr\"atzer]{gratzer@me.com}
\urladdr[G. Gr\"atzer]{http://server.math.umanitoba.ca/~gratzer/}

\author[H. Lakser]{H. Lakser}
\email[H. Lakser]{hlakser@gmail.com} 
\address{Department of Mathematics\\University of Manitoba\\Winnipeg, MB R3T 2N2\\Canada}

\date{\today}

\begin{abstract}
Let the finite distributive lattice $D$ be isomorphic
to the congruence lattice of a finite lattice $L$. 
Let $Q$ denote those elements of $D$ 
that correspond to principal congruences under this isomorphism.
Then $Q$ contains $0,1 \in D$ 
and all the join-irreducible elements of $D$. 
If $Q$ contains exactly these elements, 
we say that $L$ is a minimal representations of $D$
by principal congruences of the lattice $L$.

We characterize finite distributive lattices $D$
with a minimal representation by principal congruences
with the property that $D$ has at most two dual atoms.
\end{abstract}

\maketitle

\section{Introduction}\label{S:Introduction}

\subsection{The problem}\label{S:problem}
Let $\Princ L$ denote the ordered set of principal congruences
of the finite lattice $L$. Then 
\begin{equation}\label{E:Princ}
   \Princ L \ce \set{\zero, \one} \uu \Ji(\Con L),
\end{equation}
since the congruences $\zero, \one$ are principal 
(the congruence $\one$ is principal since $L$ is finite, hence bounded)
and the join-irreducible congruences 
are the congruences generated by prime intervals, 
and therefore principal. Let 
\begin{equation}\label{E:Min}
   \Min L = \set{\zero, \one} \uu \Ji(\Con L).
\end{equation}
Now \eqref{E:Princ} and \eqref{E:Min} combine: 
\begin{equation}\label{E:comb}
   \Min L \ci \Princ L.
\end{equation}

Let us say that a finite lattice $L$ has a 
\emph{minimal set of principal congruences}~if 
we have equality in \eqref{E:comb}, that is,
\begin{equation}\label{E:def}
   \Princ L = \Min L
\end{equation}
and we call $L$ a \emph{minimal representation of the distributive lattice $D=\Con{L}$}. If $P = \Ji(D)$, we equivalently say that $L$ is a \emph{minimal representation of the ordered set $P$}.
In the paper G. Gr\"atzer \cite{gG13}, 
we formulated the following question.

\smallskip

\tbf{Problem 4 of \cite{gG13}.}
Let $D$ be a finite distributive lattice. 
Under what conditions does $D$ have a minimal representation?
\smallskip

\subsection{Two illustrations}\label{S:illustrations}

We provide two examples. 
The first one is from G.~Gr\"atzer and H. Lakser \cite{GLa}.

\begin{illustration}\label{I:B3}
The eight element Boolean lattice $\SB 3 = \SC 2^3$
has no minimal representation.
\end{illustration}          

See \cite{GLa} for a proof.
Basically, if the lattice $L$ is a minimal representation, 
then on any maximal chain in $L$, 
we find two adjacent prime intervals generating distinct
atoms of $\Con L$. 
The two intervals together form an interval that generates
the join of two atoms of $\Con L$, contradicting minimality.

\begin{illustration}\label{I:C3}
The nine element distributive lattice $D = \SC 3^2$
has a minimal representation.
\end{illustration}          

We take $\SN 6$ (see the first diagram of Figure~\ref{F:N6}) 
as a minimal representation 
of the chain $\SC 3$. 
Then the glued sum $\SN 6 \gsum \SN 6$ 
(see the second diagram of Figure~\ref{F:N6})
is a congruence representation of $D = \SC 3^2$ 
but it is not minimal;
indeed, $\con{a,b}, \con{b,c} < \con{a,c}$
so  $\con{a,b} \jj \con{b,c} = \con{a,c}$
is principal and joint-reducible.
The third diagram of Figure~\ref{F:N6}
provides a minimal representation of $D = \SC 3^2$.

\begin{figure}[htb]
\centerline{\includegraphics[scale=1]{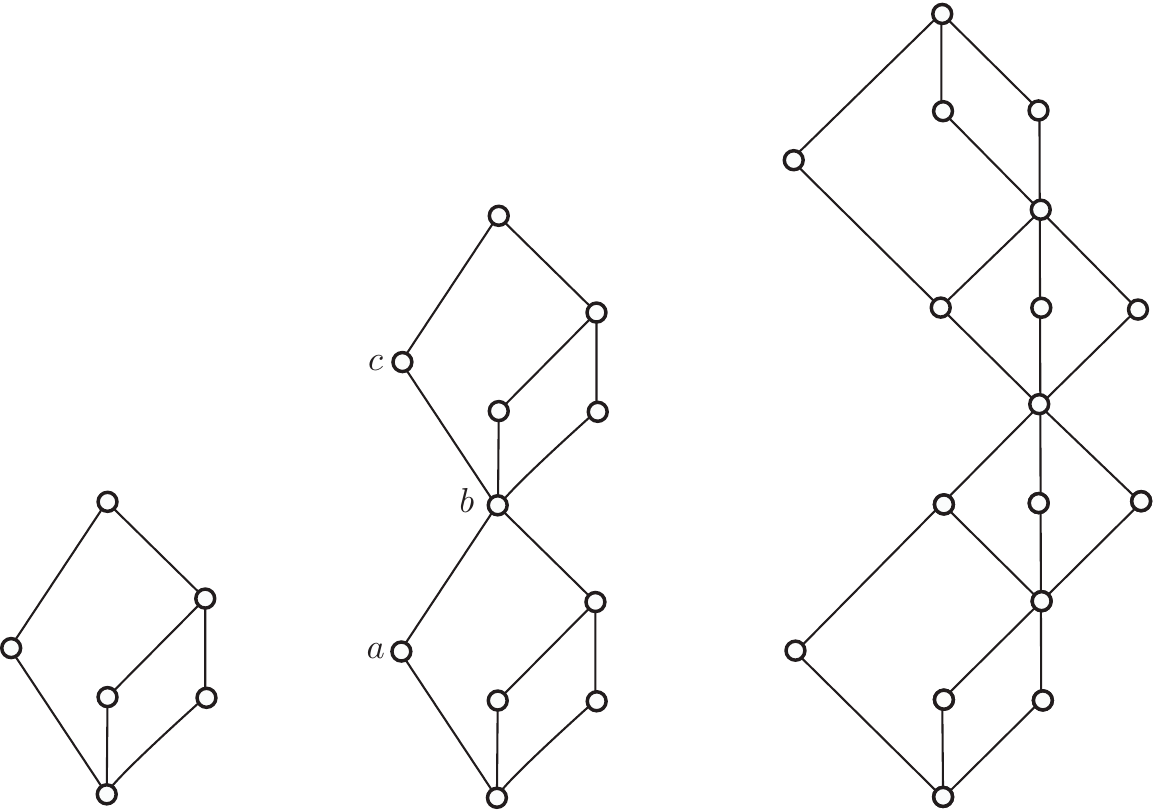}}
\caption{The lattices $\SN 6$ and its glued sum with itself }
\label{F:N6}
\end{figure}

\subsection{The result}\label{S:result}
We solve Problem 4 of \cite{gG13} as follows.

\begin{theorem}\label{T:main}
Let $D$ be a finite distributive lattice. 
Then $D$ has a minimal representation $L$
if{}f $D$ has at most two dual atoms.
\end{theorem}

Note that if a finite distributive lattice
has a minimal representation,
then it has a minimal representation of length $10$. 

This results naturally splits up into three statements.

\begin{theorem}\label{T:mainvar}
Let $D$ be a finite distributive lattice.
\begin{enumeratei} 
\item Let $D$ have exactly one dual atom. 
Then $D$ has a minimal representation~$L$.

\item Let $D$ have exactly two dual atoms.
Then $D$ has a minimal representation~$L$.

\item Let $D$ have three or more dual atoms. 
Then $D$ does not have a minimal representation $L$.
\end{enumeratei} 
\end{theorem}

The first statement is discussed in Section~\ref{S:Related}. 
The third we verify in Section~\ref{S:Three}.
The rest of the paper deals with the second statement.

\subsection{Related results}\label{S:Related}

This paper continues G.~Gr\"atzer \cite{gG13}
(see also \cite[Section 10-6]{LTF} and \cite[Part VI]{CFL2}),
whose main result is the following statement. 

\begin{theorem}\label{T:bounded}
Let $P$ be a bounded ordered set.
Then there is a bounded lattice~$K$ such that $P \iso \Princ K$.
If the ordered set $P$ is finite, 
then the lattice $K$ can be chosen to be finite.
\end{theorem}

The bibliography lists a number of papers related to this result. 

G. Gr\"atzer \cite{gGb} states (Corollaries 15 and 16) 
that the lattice $L$
for Theorem~\ref{T:bounded} constructed in G.~Gr\"atzer \cite{gG13}
(and also the one in G. Gr\"atzer \cite{gGb})
provides a minimal representation, verifying
Theorem~\ref{T:mainvar}(i).

There is a related concept.
Let us call a finite distributive lattice $D$ 
\emph{fully representable}, if every $Q \ci D$ 
satisfying $\Min L \ci Q$ is representable.

G.~Cz\'edli \cite{gCb} and \cite{gCc} prove the following result.

\begin{theorem}
A finite distributive lattice $D$ is 
fully principal congruence representable 
if{}f it is planar and 
it has at most two dual atoms 
of which at most one is join-reducible.  
\end{theorem}

It is interesting that both full 
representability and minimal representability
are determined at the dual atom level.

\subsection{Notation}
We use the notation as in \cite{CFL2}.
You can find the complete

\emph{Part I. A Brief Introduction to Lattices} and  
\emph{Glossary of Notation}

\noindent of \cite{CFL2} at 

\verb+tinyurl.com/lattices101+

\section{Three or more dual atoms}\label{S:Three}
We begin with the following result.

\begin{lemma}\label{L:corr}
For any finite distributive lattice $D$ there is a one-to-one correspondence between the
set of dual atoms of $D$ and the set of maximal elements of the ordered set $\Ji(D)$.
\end{lemma}
\begin{proof}
By the Birkoff representation theorem for finite distributive lattices, to each dual atom $a$
of $D$ corresponds a unique $p \in \Ji(D)$ with $p \nleq a$, which is perforce maximal.
The inverse correspondence assigns to each maximal element of $\Ji(D)$ the join
of all the other elements of $\Ji(D)$, which is a dual atom of $D$.
\end{proof}

In the proof of Theorem~\ref{T:mainvar}(iii) we use the following result.

\begin{theorem}\label{T:principal}
Let $L$ be a finite lattice 
and $x < y$ be elements of $L$. 
Let $A$ be an antichain of size at least $2$ 
of join-irreducible congruences of $L$ with $\JJ{A} = \con{x,y}$. 
Then for each $\bga \in A$, 
there is a join-irreducible congruence~$\bgb$ on $L$
such that the congruence $\bga \jj \bgb$ 
is principal and join-reducible.
\end{theorem}

\begin{proof}
Let 
\[
   x=c_0 \prec c_1 \prec \dots \prec c_n=y
\] 
be a (maximal) chain $C$ in the interval $[x,y]$  
and let $\bgb_i = \con{c_i, c_{i+1}}$ for $0 \leq i < n$
Then $\bgb_i$ is a join-irreducible congruence of $L$ and
\[
   \JJm{\bgb_i}{0 \leq i < n} = \con{x,y} = \JJ{A}.
\]

Let $J \ci \set{0, 1, \dots n-1}$ so that
$\setm{\con{c_i, c_{i+1}}}{i \in J}$ 
are the maximal elements in the ordered set
$\setm{\con{c_i, c_{i+1}}}{0 \leq i < n}$.
Then 
\[
   A = \setm{\con{c_i, c_{i+1}}}{i \in J},
\] 
see for instance Corollary 111 in \cite{LTF}.

So let $\bga = \con{c_j, c_{j+1}}$. 
Let $[c_k, c_l]$, with $k \leq j$ and $j+1 \leq l$, 
be a maximal subinterval of $C$ 
with $\con{c_k, c_l} = \bga$.
We cannot have both $k = 0$ and $l = n$, 
otherwise, $\bga = \JJ A$, 
contradicting the assumptions on $A$.
Without loss of generality, 
let $l < n$ and define $\bgb = \con{c_l, c_{l+1}}$. 
By the definition of $l$,
it follows that $\bgb \nleq \bga$.
So $\bgb$ is join-irreducible and $\bga \jj \bgb$ is principal
(indeed, $\bga \jj \bgb = [c_k, c_{l+1}]$).
\end{proof}

\begin{corollary}\label{C:antichain}
Let $D$ be a finite distributive lattice 
with an antichain $A$ of join-irreducible elements 
with at least $3$ elements. 
If $\JJ{A} = \one$, then $D$ does not have a minimal representation.
\end{corollary}

\begin{proof}
Assume that  the finite lattice $L$ 
with bounds $o$ and $i$ provides 
a minimal representation of $D$, 
that is, $\Princ{L} = \set{\B{0}, \one} \uu \Ji(\Con{L})$ 
and there is an isomorphism between $D$ and $\Con{L}$. 
Let $Q \ci D$ correspond to $\Princ{L}$
under this isomorphism.

By Theorem~\ref{T:principal} applied to the interval $[o,i]$, 
for each $a \in A$ there is a $b \in \Ji(D)$, 
such that $a \jj b$ is a join-reducible element of $Q$. 
Since $A$ has at least $3$~elements, 
$a \jj b \neq \one$, a contradiction.
\end{proof}

\begin{corollary}
Let $D = \SB 3$, the eight element Boolean lattice. 
Then $D$ does not have a minimal representation.
\end{corollary}

This corollary is Theorem 4 in G. Gr\"atzer and H. Lakser \cite{GLa}.

From Corollary~\ref{C:antichain} we get the following.

\begin{theorem}\label{T:three}
Let $D$ be a finite distributive lattice
with more than two dual atoms. 
Then $D$ does not have a minimal representation.
\end{theorem}

\begin{proof}
Let $M$ be the set of maximal elements of $\Ji(D)$. Then $M$ is an antichain in $D$
and $\JJ{M} = \one$. By Lemma~\ref{L:corr}, $M$ has at least three elements. Thus by
Corollary~\ref{C:antichain}, $D$ does not have a minimal representation.
\end{proof}

We have thus proved Theorem~\ref{T:mainvar}(iii). 

\section{Exactly two dual atoms, the construction}\label{S:TwoConstr}

\subsection{Preliminaries}
\label{S:preliminaries}

We will need the Technical Lemma for Finite Lattices, 
see G. Gr\"atzer \cite{gG14b}.

\begin{lemma}\label{L:technical}
Let $L$ be a finite lattice. 
Let $\bgd$ be an equivalence relation on $L$
with intervals as equivalence classes.
Then $\bgd$ is a congruence relation if{}f the following condition:
\begin{equation}\label{E:cover}
\text{if $x$ is covered by $y,z \in L$ 
and $x \equiv y\,(\textup{mod}\, \bgd)$,
then $z \equiv y \jj z\,(\textup{mod}\,\bgd)$}\tag{C${}_{\jj}$}
\end{equation}
and its dual holds.
\end{lemma}

\subsection{The construction}
\label{S:construction}
Let $P$ be a finite ordered set.
Our construction is based on the \emph{frame lattice}, 
$\Frame P$, of G.~Gr\"atzer \cite{gG13},
see Figure~\ref{F:newframe}
with the chains $\SC p = \set{0, a_p, b_p, 1}$ for $p \in P$. 
See \cite{gG13} for a detailed description; 
the diagram should suffice.

Then we consider the lattice $W(p, q)$ 
for $p < q \in P$ introduced in G. Gr\"atzer \cite{gGb},
see Figure~\ref{F:W}.
(Note that the lattice $S(p, q)$ 
used in \cite{gG13} would cause difficulties 
in the present construction.)
For all $p < q \in P$,
we insert $W(p, q)$ into $\Frame P$,
to form $\Framew P$, \emph{the frame lattice with $W$}, 
see Figure~\ref{F:Frame+W} for an illustration. 

Let $D$ be a finite distributive lattice 
with exactly two dual atoms and let $P = \Ji(D)$.
By Lemma~\ref{L:corr}, 
the ordered set $P$ has exactly two maximal elements, $p_0, p_1$.
Let $P_0 = \Dg{p_0}$ and $P_1 = \Dg{p_1}$.

Let $L_0$ be the lattice $\Framew P_0$, with zero $o$ and unit $i$, and let $L_1$, with zero $i'$ and unit $o'$, be the dual of the lattice $\Framew P_1$, where we denote by $x'$ that element of $L_1$ corresponding to $x \in \Framew{P_1}$ under the duality.

Now we are ready to construct the \emph{base lattice for $P$}, 
$\Base P$, as $L_0 \gsum L_1$,
see Figure~\ref{F:Base}. That is, $\Base{P} = L_0 \uu L_1$, with $i$ identified with $i'$ and with $x \leq i = i' \leq y$ for $x \in L_0$ and $y \in L_1$.

\begin{figure}[!t]
\centerline{\includegraphics[scale=.6]{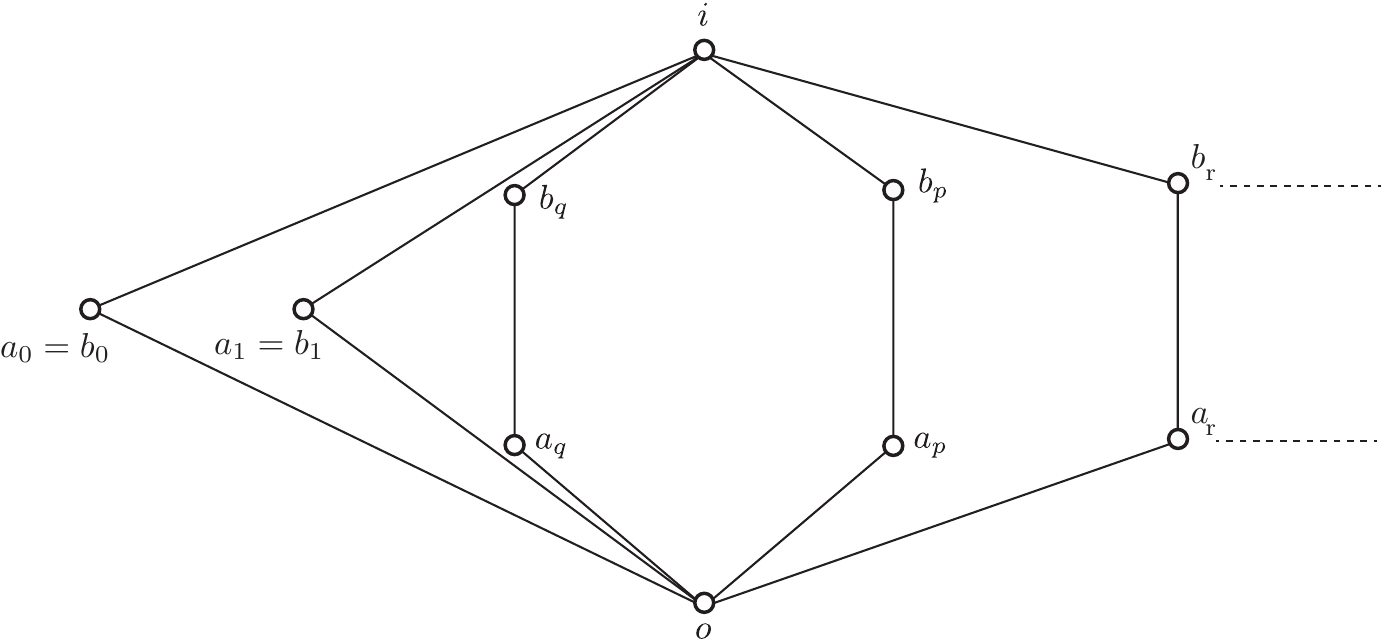}}
\caption{The frame lattice, $\Frame P$, 
with the chain $\SC p$ for $p \in P$}
\label{F:newframe}
\end{figure}

\begin{figure}[p]
\centerline{\includegraphics[scale=0.6]{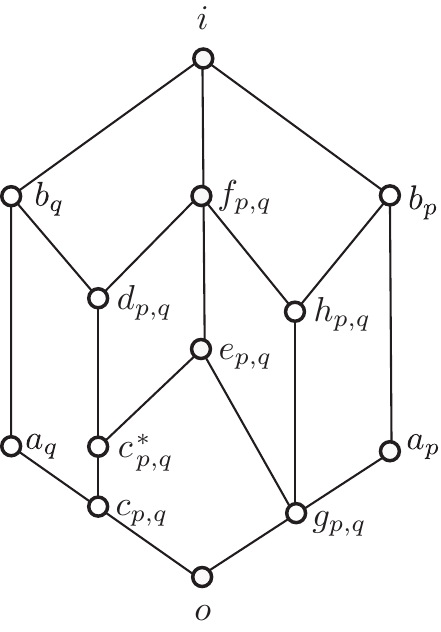}}
\caption{The lattice $W(p, q)$ for $p < q \in P$}
\label{F:W}

\bigskip

\centerline{\includegraphics[scale=.6]{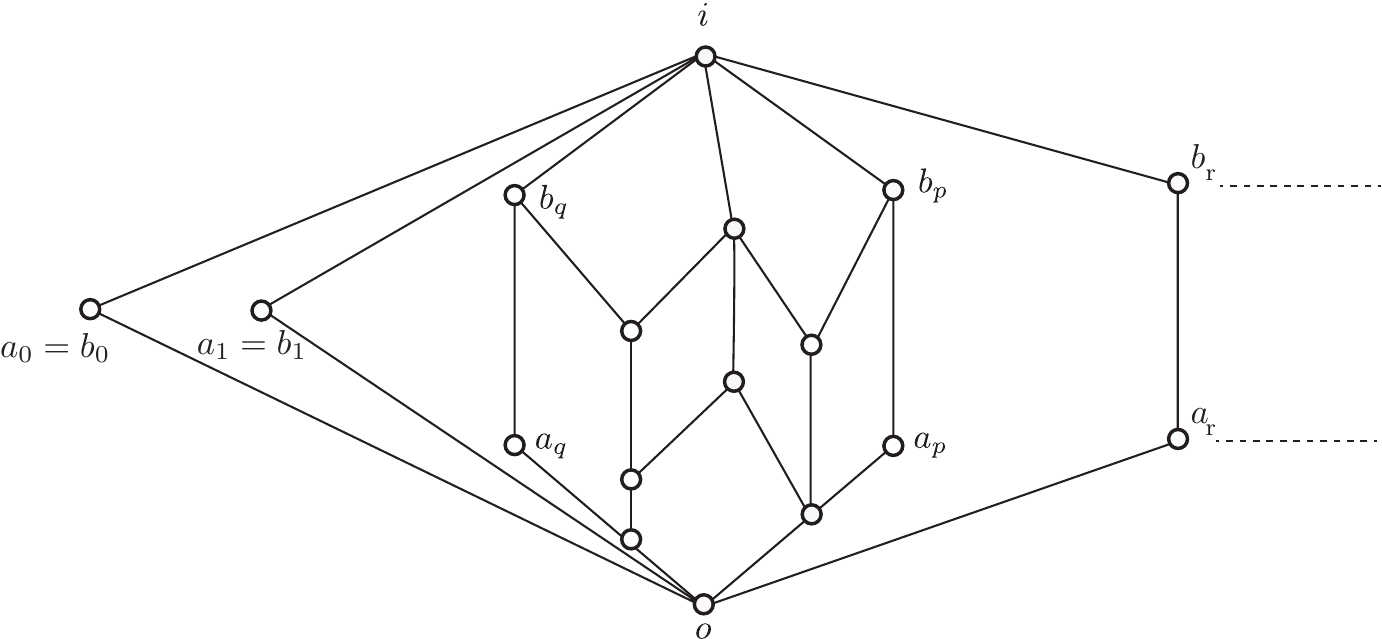}}
\caption{$\Framew P$: adding $W(p,q)$ to $\Frame P$ for $p < q \in P$}
\label{F:Frame+W}

\bigskip
\centerline{\includegraphics[scale=.69]{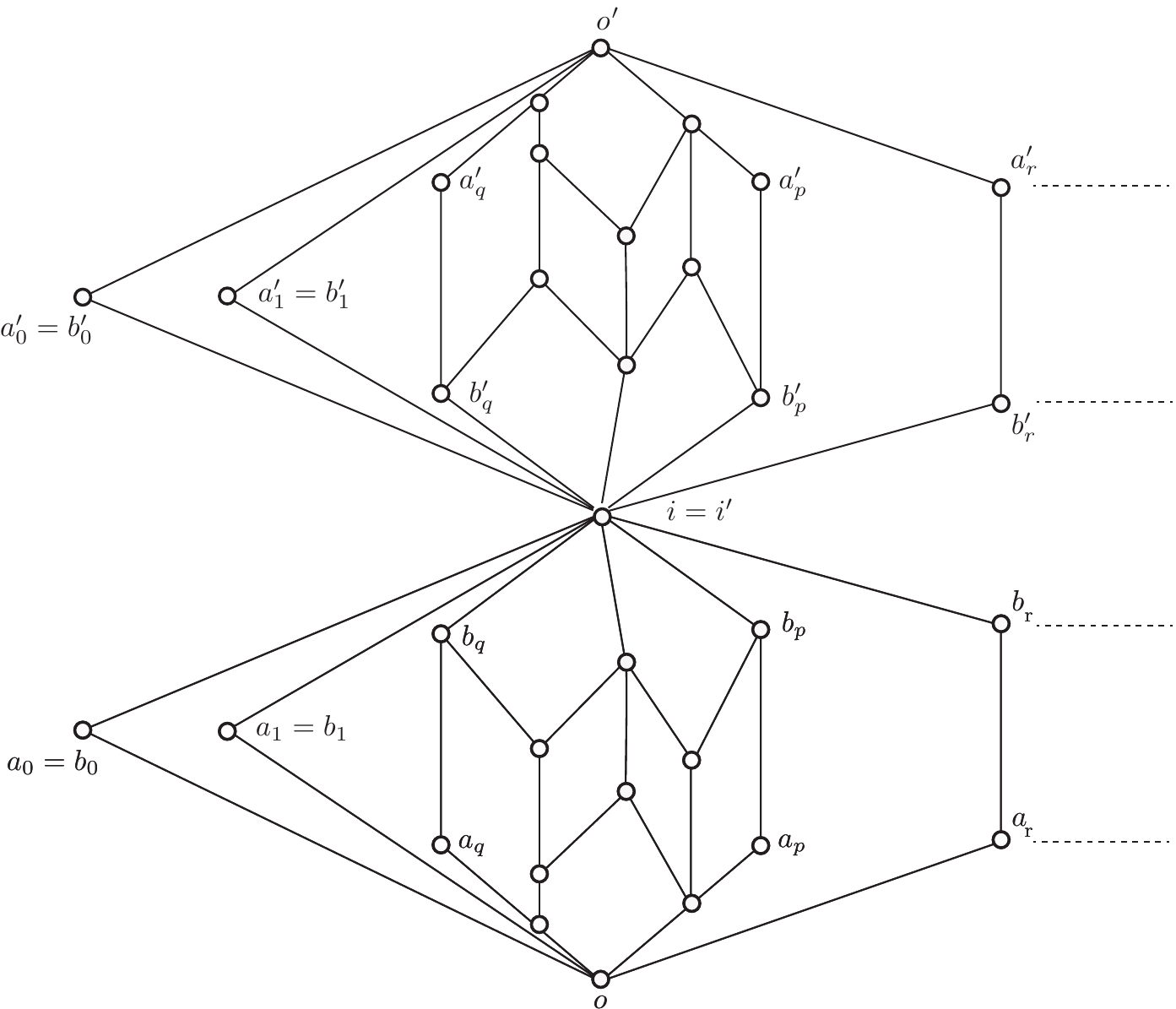}}

\caption{The base lattice of $P$, $\Base P$}
\label{F:Base}
\end{figure}

\begin{figure}[!t]
\centerline{\includegraphics[scale=.7]{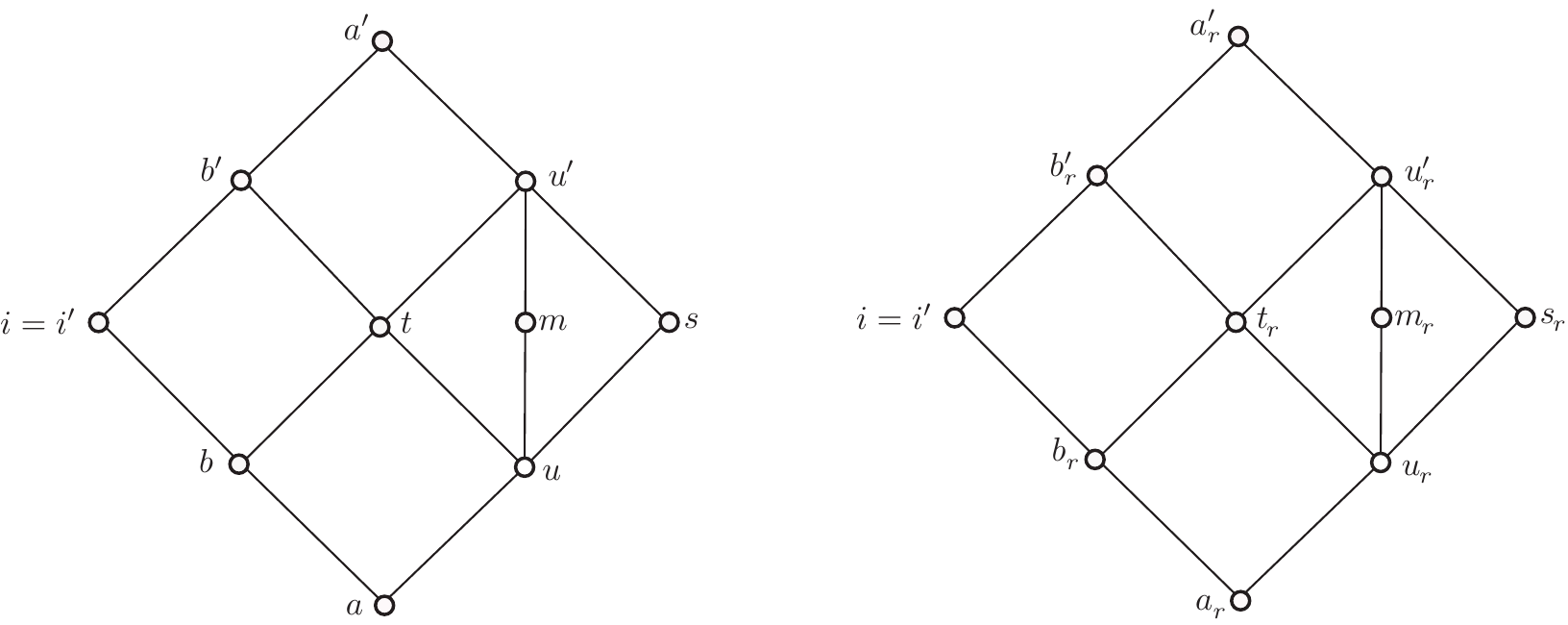}}
\caption{The bridge, $\Bridge$, 
and the $r$-bridge, $\Bridge(r)$ for $r \in P$}
\label{F:bridge}

\bigskip

\bigskip

\centerline{\includegraphics[scale=.8]{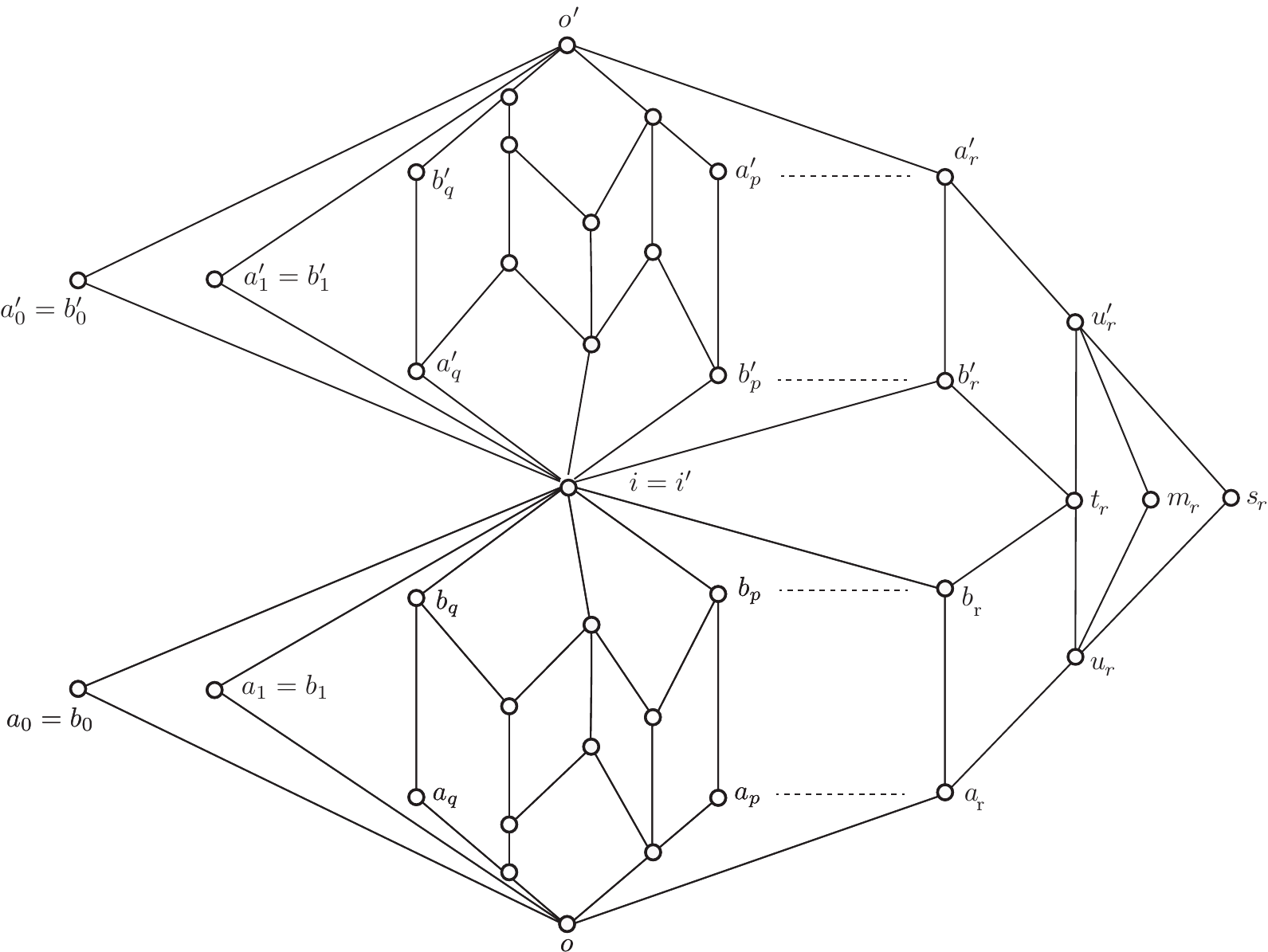}}
\caption{The base lattice of $P$, $\Base P$, 
with the $r$-bridge, $\Bridge(r)$}
\label{F:addbridge}
\end{figure}

We then show that $\Base{P}$ is a minimal representation of the ordered set $P_0 \dot{\cup} P_1$, the \emph{free union} of $P_0$ and $P_1$ with $p \incomp q$ for $p \in P_0$ and $q \in P_1$. This is easy since $\con{x,y} = \one_{\Base{P}}$ if $x < i <y$.

Now each element $r \in P_0 \ii P_1$ determines two distinct congruences of $\Base{P}$, one $\con{a_r, b_r}$, due to the sublattice $L_0$, and the other $\con{a'_r,b'_r}$, due to the sublattice $L_1$. Our main task will be to identify these two congruences, which we do with a \emph{bridge construction}.

A \emph{bridge lattice}, $\Bridge$,
is $\SC 3^2$ with an additional element $m$,
turning the covering square of the right corner into an $\SM 3$, 
see the first diagram in Figure~\ref{F:bridge}.
An \emph{$r$-bridge lattice}, $\Bridge(r)$, for $r \in P_0 \ii P_1$,
is a bridge lattice with the elements subscripted with $r$,
see the second diagram of Figure~\ref{F:bridge}.

We then obtain the desired lattice $L$ 
for Theorem~\ref{T:mainvar}(ii)
by adding a bridge for each $r \in P_0 \ii P_1$ to the base lattice $\Base P$
by forming the disjoint union of $\Base P$ and $\Bridge(r)$,
and then identifying the five elements 
$a_r, b_r, i = i', b_r', a_r'$, see~Figure~\ref{F:addbridge}. We, of course, first must show that adding a bridge results in a lattice. Then it is clear that adding a bridge will identify $\con{a_r, b_r}$ with $\con{a'_r, b'_r}$ for each $r \in P_0 \ii P_1$. Our major task will then be to show that no other congruences collapse and that all principal congruences distinct from $\one_L$ remain join-irreducible. We do this in the remainder of the paper.

\subsection{$L$ is a lattice}
\label{S:lattice}

In Sections~\ref{S:fus}--\ref{S:theorem}, 
we present the computations
showing that~$L$ provides a minimal representation for $D$,
as stated in Theorem~\ref{T:mainvar}(ii). We conclude the present section by pointing out that $L$ is, indeed, a lattice.

We start with the frame lattice $\Frame P$, 
which is obviously a lattice.
In the next step, we add $W(p,q)$ to $\Frame P$, 
for $p < q \in P$, to obtain $\Framew P$.
It~was proved in G. Gr\"atzer \cite{gGb} that $\Framew P$
is a lattice (see also G.~Gr\"atzer \cite{gG13}).
Now we need an easy statement from the folklore:

\begin{lemma}\label{L:threecover}
Let $K$ be a lattice and let $a \prec c \prec b$ in $K$.
Let $K^+ = K \uu \set{u}$ 
and define $u \mm b = a$ and $u \jj b = c$. Then
$K^+$ is a lattice extension of $K$ and,
for $x \in K$,
\[
   u \jj x = 
      \begin{cases}
      u &\text{for $x \leq a$;}\\
      b \jj x, &\text{otherwise,}
      \end{cases}
\]
and dually.
\end{lemma}

We apply Lemma~\ref{L:threecover} 
five times for each $r \in R$, that is, we add $t_r$, $u_r$, $u'_r$, $s_r$, and $m_r$ successively, to conclude that 
the lattice $L$ obtained for Theorem~\ref{T:mainvar}(ii) 
in the previous section is a lattice.

\section{Fusion and splitting in ordered sets}\label{S:fus}

We present two constructions on ordered sets that will enable us to apply the Bridge Theorem.

The first construction, \emph{fusion}, is applicable to any ordered set. The second construction, \emph{splitting}, is applicable only to ordered sets of a very special kind---including those that occur here.

\subsection{Fusion}\label{S:fusion}
Let $P$ be an arbitrary ordered set, and let $A$ be a nonempty convex subset of $P$. We define an ordered set $\Fuse(P,A)$ that is obtained in a natural manner by fusing the subset $A$ to a single element $\Fused_A$; if there is no danger of confusion,
we write $\Fused$ for $\Fused_A$.  
That is, we let 
\[
   F= \Fuse(P,A) = (P-A)\uu\set{\Fused}
\] 
and define an order on $F$. We work with the strict order $<_{F}$ rather than $\leq_{F}$, to make the definition easier to state.
For $x \in P-A$, we set
\begin{equation}\label{E:ordA}
\Fused <_F x \text{ if } a <_P x \text{ for some } a \in A,
\end{equation}
\begin{equation}\label{E:dordA}
x <_F \Fused  \text{ if } x <_P a \text{ for some } a \in A,
\end{equation}
and, for $x$, $y \in P-A$,
\begin{equation}
x <_F y 
\begin{cases}
\text{ if } x <_P y\label{E:ord}\\
\text{ or if } x <_P a_1 \text{ and } a_2 <_P y 
\text{ for some } a_1, a_2 \in A.
\end{cases}  
\end{equation}

We define a mapping $\gy_A \colon P \to \Fuse(P,A)$ by setting $\gy_A(a)= \Fused$ for $a \in A$
and $\gy_A(x) = x$ for $x \nin A$.

\begin{lemma}\label{L:fus}
The relation $<_F$ is a strict order relation on $F =  \Fuse(P,A)$ and $\gy_A \colon P \to \Fuse(P,A)$ is an isotone map.
\end{lemma}
\begin{proof}
Clearly, $<_F$ is antireflexive.

We first show that $<_F$ is antisymmetric. Let $x \neq y \in F$; we show that $x <_F y$ and $y <_F x$ cannot both hold.

First, if $x = \Fused$, then $y \in P-A$, and there exist $a_1 \in A$ with $a_1 <_P y$ and $a_2 \in A$ with $y <_P a_2$. Since $y \nin A$, these contradict the convexity of $A$.

If $y = \Fused$, we have the same argument with $x$ and $y$ interchanged.

We are then left with the cases $x \neq \Fused$ and $y \neq \Fused$. Now if $x <_P y$, then by the antisymmetry of $<_P$, we cannot have $y <_P x$. Then there are $a_1, a_2 \in A$, with $y <_P a_1$ and $a_2  <_P x$, that is, with $a_2 <_P x <_P y <_P a_1$, contradicting the convexity of $A$, since $x, y \nin A$.

 If $y <_P x$, we just exchange the roles of $x$ and $y$. 

We are finally left with the case where $x \nless_P y$ and $y \nless_P x$. Then there are $a_1, a_2, a_3, a_4 \in A$ with $x <_P a_1, a_2 <_P y$, $y <_P a_3$, and $a_4 <_P x$. Then for instance, $a_4 <_P x <_P a_1$, again  contradicting the convexity of $A$.

Consequently, $<_F$ is antisymmetric.

Finally, we establish transitivity.
So let $x, y, z \in F$ with $x <_F y <_F z$.

We first consider the case $x = \Fused$. Then $y, z \nin A$ and there is an $a \in A$ with $a <_P y$. By the convexity of $A$, $y <_F z$ cannot follow from the second case in~\eqref{E:ord}. Then $y <_P z$ and so $a <_P z$, that is, $x= \Fused <_F z$.

If $z = \Fused$, we use the dual argument.

If $y = \Fused$, then by \eqref{E:ordA}, \eqref{E:dordA}, and the second case in \eqref{E:ord}, 
we get $x <_F z$.

If one of $x, y, z$ is $\Fused$, we then have transitivity. 
So let $x, y, z$ all differ from $\Fused$. If~$x <_P y$ and $y <_P z$, then $x <_P z$, and so $x <_F z$. On the other hand, if, say, $x \nless_P y$, then the second case of \eqref{E:ord} holds, and so $x <_F \Fused <_F y < _F  z$. By transitivity whenever one of the three entries is $\Fused$, we get first $\Fused <_F z$, and then $x <_F z$. Similarly, $x <_F z$ if $y \nless_P z$. 

Thus transitivity has been established, concluding the proof that $<_F$ is an order relation.
  
It is immediate that $\gy_A$ is isotone. Indeed, let $x, y \in P$ with $x <_P y$. If $x, y \in A$, then $\gy_A(x) = \Fused = \gy_A(y)$. If $x \in A$ and $y \nin A$, then by \eqref{E:ordA}, $\gy_A(x) = \Fused <_F y = \gy_A(y)$. Similarly, $\gy_A(x) <_F \gy_A(y)$ if $y \in A$ and $x \nin A$. If $x, y \nin A$, then by the first case of \eqref{E:ord}, $\gy_A(x) = x <_F y = \gy_A(y)$.
\end{proof}

The ordered set $\Fuse(P,A)$ is the "freest" ordered set with the convex subset $A$ fusing to a single element. This is formalized by the following Universal Mapping Property.

\begin{lemma}\label{L:ump}
Let $P, Q$ be ordered sets, let $A$ be a nonempty convex subset of~$P$, and let $\gf \colon P \to Q$ be an isotone map with $\gf(a_1) = \gf(a_2)$ for all $a_1$ $a_2 \in A$. Then there is an isotone map $\gf' \colon \Fuse(P,A) \to Q$ with
\begin{equation}\label{E:ump}
\gf'\gy_A = \gf,
\end{equation}
and $\gf'$ is determined uniquely by \eqref{E:ump}.
\end{lemma} 
\begin{proof}
Let $a \in A$ be arbitrary. Then we must have $\gf'(\Fused) = \gf(a)$. For $x \nin A$, we must have $\gf'(x) = \gf(x)$.
So the isotone property for $\gf'$ 
follows from \eqref{E:ordA}, \eqref{E:dordA} and~\eqref{E:ord}. 
\end{proof}

\begin{lemma}\label{L:isom}
Let $P$ and $Q$ be ordered sets, and let $A$ be a nonempty convex subset of $P$. Let $\gf \colon P \to Q$ be a surjective isotone map with $\gf(a_1)=\gf(a_2)$ for all $a_1, a_2 \in A$.  Assume that, for all $x, y \in P$ with $x \nleq_P y$ and with $\gf(x) \leq \gf(y)$, there are $a_1, a_2 \in A$ with $x \leq_P a_1$ and $a_2 \leq_P y$. Then the isotone map $\gf' \colon \Fuse(P,A) \to Q$ determined by the condition $\gf'\gy_A = \gf$ is an isomorphism.
\end{lemma}
\begin{proof}
As before, we set $F = \Fuse(P,A)$. The map $\gf'$ is surjective since the map $\gf$ is. Thus we need only show that, for $x, y \in F$, whenever $\gf'(x) \leq \gf'(y)$, then $x \leq_F y$. We may assume that $x, y$ are distinct.

We first consider the case $x = \Fused$. Let $a$ be any  element of $A$. Then $\gf'(x) = \gf(a)$ and $\gf'(y) = \gf(y)$.   So either $a \leq_P y$ or there is an $a_2 \in A$ with $a_2 \leq_P y$. In either event, we have $\Fused \leq_F y$.

A similar argument applies if $y = \Fused$: either $x \leq_P a$ or there is an $a_1 \in A$ with $x \leq_P a_1$. Then $x \leq_F \Fused$.

Otherwise, $x, y \in P-A$. Then $\gf'(x) = \gf(x)$ and $\gf'(y) = \gf(y)$. So $\gf(x) \leq \gf(y)$. If $x \leq_P y$, then $x \leq_F y$. If $x \nleq_P y$, then by our assumption on $\gf$, there are $a_1, a_2 \in A$ with $x \leq_P a_1$ and $a_2 \leq_P y$. Then again, $x \leq_F y$, concluding the proof.
\end{proof}

\subsection{Splitting}\label{S:Splitting}
We now turn to splitting. Let $P = P_0 \uu P_1$ be an ordered set where $P_0, P_1$ are downsets of $P$, neither a subset of the other. 
Let us assume that the subset $P_0 \ii P_1$ has 
a maximal element $a$.

We then split $a$ into two incomparable elements $a_0, a_1$, as follows.

Set $S = \Split(P,a) = (P-\set{a}) \uu \set{a_0,a_1}$. Define the (strict) order $<$ on $S$ by setting, for $j = 0, 1$,
\begin{equation}\label{E:Sa}
x <_S a_j \text{ if } x \in P-\set{a} \text{ and } x <_P a,
\end{equation}
and
\begin{equation}\label{E:dSa}
a_j <_S x \text{ if } x \in P_j -\set{a} \text{ and } a <_P x,
\end{equation}
and by setting
\begin{equation}\label{E:S}
x <_S y \text{ if } x,y \in P-\set{a} \text{ and } x <_P y.
\end{equation}
We define the mapping $\gh_a \colon \Split(P,a) \to P$ by setting $\gh_a \colon a_j \mapsto a$ for $j = 0, 1$ and $\gh_a \colon x \mapsto x$ if $x \neq a_0, a_1$.

\begin{lemma}\label{L:spl}
The relation $<_S$ is a strict order relation on $S = \Split(P,a)$ and $\gh_a \colon \Split(P,a) \to P$ is an isotone map.
\end{lemma}
\begin{proof}
Clearly, $<_S$ is antirelexive.

We first show that $<_S$ is antisymmetric. Let $x \neq y \in S$; we show that $x <_S y$ and $y  <_S x$ cannot both hold. This is clear if both $x, y \nin \set{a_0, a_1}$. So without loss of generality, we may assume that $x = a_j$ for $j = 0$ or $1$, and that $y \in P-\set{a}$. Then by \eqref{E:dSa}, $a <_P y$, and, by \eqref{E:Sa}, $y <_P a$, an impossibility. Thus $<_S$ is antisymmetric.

We now establish transitivity; let $x, y, z \in S$ with $x <_S y <_S z$. If none of $x, y$, or $z$ is an element of $\set{a_0, a_1}$, then $x <_P y <_P z$ by \eqref{E:S}, and so $x <_P z$, that is, $x <_S z$.
If $y \in \set{a_0, a_1}$, then by \eqref{E:Sa} and \eqref{E:dSa}, $x, z \in P-\set{a}$ and $x <_P a  <_P z$, that is, $x <_P z$ and so $x <_S z$. If $z \in \set{a_0,a_1}$, we get $x <_P y <_P a$, and so $x <_P a$, whereby $x <_S z$.

We are then left only with the case $y, z \in P-\set{a}$ and $x = a_j$ for $j \in \set{0,1}$. Then $y \in P_j$ and $a <_P y$
 by \eqref{E:dSa}. Furthermore, $y <_P z$
  by \eqref{E:S}. Thus $a < _P z$. We now observe that we cannot have $z \in P_{1-j}$, for, in that event, we would have $y \in P_0 \ii P_1$, contradicting the maximality of $a$. Thus $z \in P_j$ and $x = a_j <_S z$ by \eqref{E:dSa}, establishing transitivity.

It is clear that $\gh_a$ is isotone 
from the definition of $<_S$.
\end{proof}

Now if $P = P_0 \uu P_1$ is as above, and $a$ is maximal in $P_0 \ii P_1$, then the subset $A = \set{a_0,a_1}$ of $\Split(P,a)$ is convex. We can then fuse $A$, getting the ordered set $\Fuse(\Split(P,a),A)$. We have the mapping 
\[
   \gh' \colon \Fuse(\Split(P,a),A) \to P
\]
with $\gh' \colon \Fused \mapsto a$ and $\gh' \colon x \mapsto x$ if $x \neq \Fused$. We have the isotone 
\[
   \gy_A \colon \Split(P,a) \to \Fuse(\Split(P,a),A),
\]
satisfying
\[
   \gh'\gy_A = \gh_a \colon \Split(p,a) \to P.
\]
We apply Lemma~\ref{L:isom} to show that $\gh'$ is an order isomorphism. We only have to show that $\gh_a \colon \Split(P,a) \to P$ satisfies the condition assumed there for $\gf$.

So assume that $x, y \in S = \Split(P,a)$ with $x \nleq_S y$ and $\gh_a(x) \leq_P \gh_P(y)$.  Then $x \in A = \set{a_0,a_1}$; 
indeed, otherwise, by \eqref{E:Sa} and \eqref{E:S}, $x = \gh_a(x) \leq_P \gh_a(y)$ implies that $x \leq_S y$. If $y \in A$ also, then $x = a_j$ and $y=a_{1-j}$, for some $j \in \set{0,1}$, establishing the hypothesis.  If $y \nin A$, then $a_j \leq_S y$, for some $j \in \set{0,1}$, by \eqref{E:dSa}, and 
$x=a_{1-j}$ since $x \nleq_S y$, again establishing the condition for $\gh_a$. We thus have:

\begin{lemma}\label{L:isoFS}
Setting $A = \set{a_0, a_1}$, the mapping 
\[
   \gh' \colon \Fuse(\Split(P,a),A) \to P,
\]
whereby $\gh' \colon \Fused \mapsto a$
and $\gh' \colon x \mapsto x$, otherwise, is an order isomorphism.
\end{lemma}

\section{Admissible congruences and extensions}
\label{S:Admissible}

Let $K$ be a finite lattice, 
and let  $a \prec c \prec b$ be three elements  of $K$. 
As in Lemma~\ref{L:threecover}, 
we extend $K$ to the lattice $K^+$ 
by adjoining a new element $u$ 
as a relative complement of $c$ in the interval $[a,b]$.

We start with an easy and well-known statement.

\begin{lemma}\label{L:det}
For the lattice $K$ above,
$K$ is a congruence-determining sublattice of $K^+$.
\end{lemma}
\begin{proof}
Let $\bga$ be a congruence relation on $K^+$. We show that $\bga$ is determined by $\bga\restr K $. We need only consider the congruence class of $u$.

Let $x > u$. Then $x \in K$ and $x \geq b$. So
$\cng u = x(\bga)$  if{}f  $\cng b = x (\bga\restr K )$  and $\cng u=b (\bga)$ if{}f $\cng b = x  (\bga\restr K )$ and $\cng a = c (\bga\restr K )$, the latter because $u$ is a relative complement of $c$ in $[a,b]$.

Dually, if $x < u$, then $\cng u = x(\bga)$  if{}f  $\cng a = x (\bga\restr K )$  and $\cng u=a (\bga)$ if{}f $\cng a = x  (\bga\restr K )$ and $\cng b = c (\bga\restr K )$.

Thus $\bga$ is indeed determined by $\bga\restr K $.
\end{proof}

We now determine which congruence relations on $K$ extend to $K^+$.

\begin{definition}\label{D:admissible}
For the lattice $K$ above, 
a congruence relation $\bga$ on $K$ is \emph{admissible}, 
if it satisfies the following four conditions.
\begin{enumeratei}
\item \label{I:con} If $x \succ a$ and $\cng x=a (\bga)$, then $\cng c=a (\bga)$.
\item \label{I:ns} If $a$ is meet-reducible in $K$ and $\cng c=b (\bga)$, then $\cng a=b (\bga)$
\item\label{I:cond} If $x \prec b$  and $\cng x=b (\bga)$, then $\cng c=b (\bga)$.
\item\label{I:nsd} If $b$ join-reducible in $K$ and $\cng a=c (\bga)$, then $\cng a=b (\bga)$.
\end{enumeratei}
\end{definition}

Note that admissibility is a self-dual concept: 
(iii) is the dual of \eqref{I:con}, 
and \eqref{I:nsd} is the dual of  \eqref{I:ns}.

\begin{theorem}\label{T:admis}
For the lattice $K$ above,
a congruence relation $\bga$ on $K$ can be extended to the lattice $L=K^+$ if{}f it is admissible.
\end{theorem}
\begin{proof}
We first assume that $\bga$ has an extension $\bgb$ to $L$. We show that $\bga$ is admissible.

We first show that \eqref{I:con} holds for $\bga$.  Let $x \in K$ with $x \succ a$ and $\cng x = a (\bga)$.  Then $\cng x = a (\bgb)$ and so 
\[
   \cng {x \jj b = x \jj u} ={ a \jj u = u} (\bgb).
\]
Taking the meet with $c$, we conclude that  $\cng c = a(\bgb)$, that is, $\cng c = a (\bga)$, thereby establishing \eqref{I:con}. 

The dual argument establishes \eqref{I:cond}.

Next we show that \eqref{I:ns} holds for $\bga$. Assume that $a$ is meet-reducible in $K$, and so there is an $x > a$  in $K$  with $x \mm c =a$. Assume, furthermore, that
\[
   \cng c=b (\bga).
\]
Then $\cng c =b(\bgb)$ (in $L$), and so $\cng u = a(\bgb)$. So
\[
\cng {b \jj x = u \jj x} = {a \jj x =x}(\bgb),
\]
that is, $\cng b \jj x = x (\bga)$, and thus
\[
\cng {c = (b \jj x) \mm c} = {x \mm c = a} (\bga).
\]
Thus $\cng a = b (\bga)$, thereby establishing \eqref{I:ns}. 

The dual argument establishes \eqref{I:nsd}.

Therefore, any congruence relation on $K$ that extends to $L$ is admissible.

Now let $\bga$ be an admissible congruence on $K$. We construct a congruence~$\bgb$ on $L$ that is an extension of $\bga$.

We first define $\bgb$ as an equivalence relation on $L$, 
and then show that it is indeed a congruence relation. 
We need only describe the equivalence class of $u$. 
If $\cng a = c (\bga)$, 
we put $u$ in the equivalence class of $b$, 
and if $\cng c = b (\bga)$ we put $u$ in the equivalence class of $a$. If both equivalences hold, there is no contradiction, 
since then $\cng a = b (\bga)$. 
If neither equivalence holds, 
we let $u$ be in its own singleton equivalence class $\set{u}$. 
Then clearly,
\[
\bgb\restr K  = \bga.
\]

Note that $\bgb$ is a self-dual equivalence relation.

In order to show that $\bgb$ is a congruence relation, we apply Lemma~\ref{L:technical}, the Technical Lemma for Finite Lattices.

We first show that all of the equivalence classes of $\bgb$ are intervals in $L$. We first consider the 
equivalence classes that do not contain $u$. 
Let $X$ be such an equivalence class. 
Then $X \ci K$ and is an interval $[x, y]_K$ for some $x \leq y$ in $K$. We claim that $X = [x,y]_L$. 
For otherwise, $u \in [x, y]_L$ and so $x \leq a < u < b \leq y$. 
It would then follow that $\cng a = b (\bga)$ 
and so by the definition of $\bgb$, 
that $\cng u=a (\bgb)$ and $\cng u = b  (\bgb)$, 
contradicting our assumption that $u \nin X$, since $a, b \in X$. 
Thus any equivalence class of $\bgb$ 
that does not contain $u$ is an interval in $L$.

Now let $X$ be the equivalence class of $\bgb$ that contains $u$. 
If $\ncng a = c (\bga)$ and $\ncng c = b (\bga)$, 
then by the definition of $\bgb$, $X = \set{u}$, 
an interval in $L$.

Otherwise, by duality, we may assume that
\[
\cng c = b (\bga).
\]
Then by definition of $\bgb$,  $\cng u = a (\bgb)$, that is, $a \in X$. Since $\bgb\restr K  = \bga$, 
it follows that $X \ii K$ is the $\bga$-congruence class of $a$,  that is, $X \ii K = [x,y]_K$ for some $x, y \in K$ with $x \leq a \leq y$. So
\[
X = [x, y]_K \uu \set{u}.
\]
If $a <y$, then there is a $z \in K$ with $a \prec z \leq y$,  
and so  $\cng z = a (\bga)$. By condition~\eqref{I:con},
$\cng a = c(\bga)$. Thus $\cng a = b (\bga)$, that is, $y \geq b$.  
Then $u \in [x ,y]_L$, and so $X = [x,y]_L$.

If, on the other hand, $a=y$, then $X = [x,u]_L$, 
since $u\succ a$ in $L$.

Consequently, all the equivalence classes of $\bgb$ are intervals in $L$.

We now verify \eqref{E:cover} of Lemma~\ref{L:technical} for $\bgb$.
Since \eqref{E:cover} hold trivially if $y = z$,
let $x, y, z$ be distinct elements of $L$ with $x \prec y,z$ and with $\cng x=y(\bgb)$. 
We show that $\cng z = y \jj z (\bgb)$.

Since $u$ is meet-irreducible, $x \neq u$.
 
If both $y, z \neq u$, then we are in $K$ with $\cng x = y (\bga)$. Then $\cng z = y \jj z (\bga)$, and so $\cng z = y \jj z (\bgb)$. Otherwise, either $z=u$ or $y=u$, and so $x = a$, the unique lower cover of $u$.

If $z=u$, then $\cng y = a (\bga)$. By \eqref{I:con}, $\cng a = c (\bga)$. Thus
\[
   \cng u =b (\bgb),
\]
by definition. Furthermore, $\cng {y \jj z = y \jj u = y \jj b} =  b (\bga)$, that is,
\[
   \cng y \jj z =  b (\bgb).
\]
So $\cng y \jj z = {u = z} (\bgb)$ since $\bgb$ is transitive.

If $y = u$, that is, if $\cng a = u(\bgb)$, 
then $\cng c = b(\bga)$ by definition of $\bgb$.

If $z=c$, then
\[
   \cng {z = c} = {b = u \jj c = y \jj z} (\bga),
\]
that is, $\cng z = {y \jj z}(\bgb)$.

If $z \neq c$, then $a$ is meet-reducible; since $\cng c = b (\bga)$, 
we conclude by \eqref{I:ns} that $\cng a = b (\bga)$. 
Then
\[
\cng {z = a \jj z} = {b \jj z = u \jj z =y \jj z} (\bga),
\]
that is, again that $\cng z = {y \jj z} (\bgb)$.

Thus for all distinct $x, y, z \in L$, with $x \prec y,z$ and with $\cng x=y(\bgb)$,
it follows that   $\cng z = {y \jj z}(\bgb)$, verifying \eqref{E:cover}. 

The dual argument verifies the dual of \eqref{E:cover}.

Consequently, by Lemma~\ref{L:technical}, $\bgb$ is a congruence relation on $L$, and by its definition, is the extension of $\bga$ to $L$.
\end{proof}

\begin{corollary}\label{C:ext}
For the lattice $K$ above,
let $\bga$ be a congruence relation on $K$
and let $L = K^+$.
\begin{enumeratea}
\item If $\bga$ is admissible, 
then $\conL{\bga}\restr K  = \bga$.
\item If $b$ is join-irreducible in $K$ 
and only \ref{D:admissible}\eqref{I:con} fails for $\bga$, 
then \[\conL{\bga}\restr K  = \bga \jj \conK{a,c}.\]
\item If $a$ is meet-irreducible in $K$ 
and only \ref{D:admissible}\eqref{I:cond} fails for $\bga$, 
then \[\conL{\bga}\restr K  = \bga \jj \conK{b,c}.\]
\item Otherwise, $\conL{\bga}\restr K  = \bga \jj \conK{a,b}$.
\end{enumeratea}
\end{corollary}
\begin{proof}
If $\bga$ is admissible, then $\conL{\bga}$ is its extension to $L$, 
and so we have $\conL{\bga}\restr K  = \bga$.

Now for any congruence $\bga$ on $K$, the congruence $\bga \jj \conK{a, b}$ is admissible.
Thus for any $\bga$,
\[
\conL{\bga}\restr K  \leq \conL{\bga \jj \conK{a,b}}_{\restr_K} = \bga \jj \conK{a,b}.
\]

Now assume that $b$ is join-irreducible, and that only \ref{D:admissible}\eqref{I:con} fails for $\bga$, that is, there is a cover $x$ of $a$ distinct from $c$ with $\cng a=x(\bga)$ and $\ncng a=c(\bga)$. It is easy to see that, in this case,  $\bga \jj \conK{a,c}$ is admissible. Thus $\conL{\bga}\restr K  = \bga \jj \conK{a,c}$.

If on the other hand, $b$ is join-reducible, then $\bga \jj \conK{a,c}$ is not admissible unless \[\cng c = b (\bga \jj \conK{a,c}),\] in which case \[\bga \jj \conK{a,c} = \bga \jj \conK{a,b}.\]

If \ref{D:admissible}\eqref{I:cond} fails for $\bga$, we are in the dual situation.

Finally, if \ref{D:admissible}\eqref{I:ns} or \eqref{I:nsd} fail, any admissible extension $\bgg$ of $\bga$  satisfies $\cng a = b(\bgg)$. Thus if \ref{D:admissible}\eqref{I:ns} or \eqref{I:nsd} fails, then $\conL{\bga}\restr K  = \bga \jj \conK{a,b}$.
\end{proof}

Although we make no use of it in this paper, observe that the Tab Lemma 
(G. Cz\'edli, G. Gr\"atzer, and H.~Lakser \cite[Lemma 12]{CGL17})
easily follows.

\begin{corollary}[\textbf{Tab Lemma}]
Let $L$ be any finite lattice and let $u$ be a tab of~$L$ in the covering multidiamond $[a, b]$. Set $K= L - \set{u}$, a sublattice of $L$. Let $\bga$ be a congruence relation on $K$, and set $\bgb = \conL{\bga}$.

If $\ncng a = b (\bgb)$, then $\bgb\restr K  = \bga$.

If $\cng a = b (\bgb)$, then $\bgb\restr K  = \bga \jj \conK{a,b}$.
\end{corollary}
\begin{proof}
Let $c$ and $c'$ be two other atoms of  the multidiamond $[a,b]$. The lattice $L$ then is obtained by adding a relative complement of $c$ in $[a,b]_K$. Thus Corollary~\ref{C:ext} applies.

The element $a$ is meet-reducible in $K$ and $b$ is join-reducible because of the element $c'$. By Corollary~\ref{C:ext},
if $\bga$ is not admissible, then $\bgb\restr K  = \bga \jj  \conK{a, b}$. Furthermore, $\cng a=b (\bgb)$. So if $\bga$ is not admissible, then $\cng a = b (\bgb)$ and $\bgb\restr K  = \bga \jj \conK{a,b}$.

On the other hand, if $\bga$ is admissible, then $\bgb\restr K  = \bga$ which is $\bga \jj \conK{a,b}$ exactly when $\cng a = b (\bga)$ exactly when $\cng a = b (\bgb)$. 

Summarizing, if $\ncng a = b (\bgb)$, then $\bga$ is admissible, and $\bgb\restr K  = \bga$, and if $\cng a = b (\bgb)$, then $\bgb\restr K  = \bga \jj \conK{a,b}$ whether or not $\bga$ is admissible.
\end{proof}

\section{Preliminaries for the bridge construction}\label{S:prelim}

We present three results that serve 
as the foundation of the bridge construction. 

\begin{lemma}\label{L:admis}
Let $K$ be a finite lattice containing the elements $a, b, c$ with $a \prec c \prec b$, such that $a$ is meet-irreducible and $b$ is join-irreducible. Let $K$ be extended to the lattice $K^+$ by adjoining a relative complement $u$ of $c$ in the interval $[a,b]$. Then $K^+$ is a congruence preserving extension of $K$
\end{lemma}
\begin{proof}
Let $\bga$ be any congruence relation on $K$. Then it is easy to see that $\bga$ is admissible. Indeed, \ref{D:admissible}\eqref{I:ns} and \eqref{I:nsd} do not apply, and \ref{D:admissible}\eqref{I:con} and \eqref{I:cond} apply trivially, since the only possible $z$ in each is $c$. Thus by Theorem~\ref{T:admis}, any congruence of $K$ extends to $K^+$. By Lemma~\ref{L:det} it follows that $K^+$ is indeed a congruence preserving extension of $K$. 
\end{proof}

Henceforth, in this section, $K$ will be a finite lattice with elements $a, c, c', b$, with
$a \prec c \prec b$, with $a \prec c' \prec b$, with $a$ having no upper cover other than $c$ and $c'$, and with $b$ having no lower cover other than $c$ and $c'$. We adjoin a relative complement $u$ of $c$ (and so also of $c'$) in the interval $[a,b]$ to get the extension $L = K^+$ of $K$. In this section, $L$ will always refer to $K^+$.

\begin{lemma}\label{L:inad}
If the congruence $\bga$ of $K$ is not admissible, then either
\[
\conK{a,c} \leq \bga \text{ and } \conL{\bga}\restr K  = \bga \jj \conK{b,c}
\] 
or
\[
\conK{b,c} \leq \bga \text{ and } \conL{\bga}\restr K  = \bga \jj \conK{a,c}.
\] 
\end{lemma}
\begin{proof}
If $\ncng a = c (\bga)$ and $\ncng b = c (\bga)$, then \ref{D:admissible}\eqref{I:ns} and \eqref{I:nsd} hold.

Now $\cng a = c' (\bga)$ if{}f $\cng b = c (\bga)$ and $\cng b = c' (\bga)$ if{}f $\cng a = c (\bga)$. Since $c'$ and $c$ are the only upper covers in $K$ of $a$ and the only lower covers in $K$ of $b$, \ref{D:admissible}\eqref{I:con} and \eqref{I:cond} also hold. Thus if $\ncng a = c (\bga)$ and $\ncng b = c (\bga)$, then $\bga$ is admissible.

So if $\bga$ is not admissible, then either $\cng a = c (\bga)$ or $\cng b = c (\bga)$. Furthermore, by Corollary~\ref{C:ext},
\[
\conL{\bga}\restr K  = \bga \jj \conK{a,b} = \bga \jj \conK{a,c} \jj \conK{b,c}.
\]
So if $\cng a = c (\bga)$, that is, if $\conK{a,c} \leq \bga$,  then \[\conL{\bga}\restr K  = \bga \jj \conK{b,c},\] and if $\cng b = c (\bga)$, that is, if $\conK{b,c} \leq \bga$, then
\[\conL{\bga}\restr K  = \bga \jj \conK{a,c},\] concluding the proof.
\end{proof}

We note the following triviality that will be useful in our calculations.

\begin{lemma}\label{L:trivial}
Let the lattice $K_0$ be a sublattice of the lattice $K_1$, and let $x, y, z, w \in K_0$.
If $\consub{K_0}{x,y} \leq \consub{K_0}{z,w}$, then $\consub{K_1}{x,y} \leq \consub{K_1}{z,w}$.
\end{lemma}
\begin{proof}
Let $\bga$ and $\bgb$ be congruence relations on $K_0$ with $\bga \leq \bgb$.
Then $\consub{K_1}{\bga} \leq \consub{K_1}{\bgb}$.
Furthermore, $\consub{K_1}{s,t} = \consub{K_1}{\consub{K_0}{s,t}}$
for $s, t \in K_0$. So for the elements $x, y, z, w \in K_0$, if
$\consub{K_0}{x,y} \leq \consub{K_0}{z,w}$, then $\consub{K_1}{x,y} \leq \consub{K_1}{z,w}$.
\end{proof}

With $K, L$, $a, b, c, c', u$ as above, we have:

\begin{lemma}\label{L:bridge}
Let $x_0, y_0, x_1, y_1 \in K$ with $x_0 \prec y_0$ in $K$  and $x_1 <  y_1$. Then
$\conL{x_0,y_0} \leq \conL{x_1,y_1}$ if{}f at least one of the following three conditions holds:
\begin{enumeratea}
\item\label{I:br1} $\conK{x_0,y_0} \leq \conK{x_1,y_1}$.
\item\label{I:br2} $\conK{x_0,y_0} \leq \conK{a,c}$ and $\conK{b,c} \leq \conK{x_1,y_1}$.
\item\label{I:br3} $\conK{x_0,y_0} \leq \conK{b,c}$ and $\conK{a,c} \leq \conK{x_1,y_1}$.
\end{enumeratea}
\end{lemma}
\begin{proof}
We first show that each of \eqref{I:br1}, \eqref{I:br2}, \eqref{I:br3} implies
\begin{equation}\label{E:Lleq}
\conL{x_0,y_0} \leq \conL{x_1,y_1}.
\end{equation}
It is immediate from Lemma~\ref{L:trivial}
that \eqref{I:br1} implies \eqref{E:Lleq}.
Since $\set{a,c,u,c',b}$ is an $\SM 3$, it follows that $\conL{a,c} = \conL{b,c}$. Then it is also immediate that each of \eqref{I:br2}, \eqref{I:br3} implies \eqref{E:Lleq}.

We now assume \eqref{E:Lleq} and show that at least one of \eqref{I:br1}, \eqref{I:br2}, \eqref{I:br3} holds. To accomplish this, we assume that \eqref{I:br1} fails:
\begin{equation}\label{E:nbr1}
\conK{x_0,y_0} \nleq \conK{x_1,y_1},
\end{equation}
and show that at least one of \eqref{I:br2}, \eqref{I:br3} holds.

By \eqref{E:Lleq}, we have 
\[\conL{x_0,y_0}\restr K  \leq \conL{x_1,y_1}\restr K,\] and so
\begin{equation}\label{E:Kleq}
\conK{x_0,y_0} \leq \conL{x_1,y_1}\restr K .
\end{equation}

Now there are three possibilities for $\conL{x_1,y_1}\restr K $. Either $\conK{x_1,y_1}$ is admissible, and so we have
\begin{equation}\label{E:1}
\conL{x_1,y_1}\restr K  = \conK{x_1,y_1},
\end{equation}
or either
\[
\conK{a,c} \leq \conK{x_1,y_1} \text{ and } \conL{x_1,y_1}\restr K  = \conK{x_1,y_1} \jj \conK{b,c}
\] 
or
\[
\conK{b,c} \leq \conK{x_1,y_1} \text{ and } \conL{x_1,y_1}\restr K  = \conK{x_1,y_1} \jj \conK{a,c}
\] 
holds by Lemma~\ref{L:inad}.
Thus by \eqref{E:Kleq}, either
\begin{equation}\label{E:2}
\conK{a,c} \leq \conK{x_1,y_1} \text{ and } \conK{x_0,y_0}\leq  \conK{x_1,y_1} \jj \conK{b,c}
\end{equation} 
or
\begin{equation}\label{E:3}
\conK{b,c} \leq \conK{x_1,y_1} \text{ and } \conK{x_0,y_0} \leq \conK{x_1,y_1} \jj \conK{a,c}.
\end{equation}
Since $x_0 \prec y_0$, $\conK{x_0,y_0}$ is a join-irreducible congruence relation. By \eqref{E:nbr1},
\eqref{E:2} implies
\[
\conK{b,c} \leq \conK{x_1,y_1} \text{ and } \conK{x_0,y_0} \leq \conK{a,c},
\]
that is, implies \eqref{I:br2}.

Similarly, \eqref{E:3} implies \eqref{I:br3}.

Thus if \eqref{I:br1} fails, then \eqref{E:Lleq} implies that either \eqref{I:br2} or \eqref{I:br3} holds. Consequently, \eqref{E:Lleq} implies that at least one of \eqref{I:br1}, \eqref{I:br2}, \eqref{I:br3} holds, concluding the proof. 
\end{proof}

\section{The Bridge Theorem}

We start with the Bridge Theorem, which examines how 
the congruences of a bridge extension behave.

\begin{theorem}\label{T:BT}
Let $K$ be a finite lattice and let $a \prec b \prec i \prec b' \prec a'$ in~$K$. We~assume that $a$ and $b$ are meet-irreducible, and $a'$ and $b'$ are join-irreducible. Let the lattice $L$ result from attaching to $K$,  between $[a,b]$ and $[b',a']$, the bridge 
\[\Bridge =\set{a,b,i,b',a',t,u,u',s,m}\] 
depicted in Figure~\ref{F:bridge}.

Then the following hold:
\begin{enumeratei}
\item\label{I:BT1} If $x \prec y$ in $K$, then $x \prec y$ in $L$.

\item\label{I:BT2}  If $x \in K$ is meet-irreducible in $K$ and it differs from $a$ and $b$, then $x$ is meet-irreducible in $L$.

\item\label{I:BT3}  If $x \in K$ is join-irreducible in $K$ and it differs from $a'$ and $b'$, then $x$ is join-irreducible in $L$.

\item\label{I:BT4}  The sublattice $L - \set{m}$ of $L$ is a congruence preserving extension of $K$.

\item\label{I:BT5}  Each join-irreducible congruence of $L$ is  of the form $\conL{x,y}$ for some $x, y \in K$ with $x \prec y$ in $K$. 

\item\label{I:BT6}  For any $x_0, y_0, x_1, y_1 \in K$ 
with $x_0 \prec y_0$ in $K$  and $x_1 <  y_1$, 
the congruence inequality \[\conL{x_0,y_0} \leq \conL{x_1,y_1}\] holds if{}f at least one of the following three conditions holds:
\end{enumeratei}

\begin{equation}\label{I:brt1} 
\conK{x_0,y_0} \leq \conK{x_1,y_1},
\end{equation}
\begin{equation}\label{I:brt2} 
\conK{x_0,y_0} \leq \conK{a,b} \text{ and }
\conK{a',b'} \leq \conK{x_1,y_1},
\end{equation}
\begin{equation}\label{I:brt3} 
\conK{x_0,y_0} \leq \conK{a',b'} \text{ and } 
\conK{a,b} \leq \conK{x_1,y_1}.
\end{equation}
\end{theorem}
\begin{proof}
We get to $L$ from $K$ by successively adjoining relative complements and observing that \eqref{I:BT1}, \eqref{I:BT2}, and \eqref{I:BT3} hold at each stage (for each new lattice, rather than just $L$). Thus \eqref{I:BT1}, \eqref{I:BT2}, and \eqref{I:BT3} are verified.

Now we first adjoin the relative complement $t$ of $i$ in $[b,b']$ to get the lattice $K_1$.  By Lemma~\ref{L:admis}, $K_1$ is a congruence preserving extension of $K$. Now $a$ is still meet-irreducible in $K_1$ and $t$ is doubly-irreducible in $K_1$. So $K_2 = K_1 \uu \set{u}$ is a congruence preserving extension of $K_1$, where $u$ is doubly-irreducible, $t$~is now meet-irreducible and $a'$ is still join-irreducible. Then $K_3 = K_2 \uu \set{u'}$ is a congruence preserving extension of $K_2$, where $u$ and $u'$ are both doubly-irreducible. Then \[K_4 = K_3 \uu \set{s} = L - \set{m}\] is a congruence preserving extension of $K_3$. Thus statement~\eqref{I:BT4} is verified.

Now the only prime intervals in $L$ that are not in $K_4$ are $[u,m]$ and $[m,u']$, and 
\[\conL{u.m} = \conL{u',m} = \conL{a,b}.\] Thus each join-irreducible congruence of $L$ is $\conL{x', y'}$ for some $x' \prec y'$ in $K_4$. But $K_4$ is a congruence preserving extension of $K$; thus
$\consub{K_4}{x',y'} = \consub{K_4}{x,y}$ for some $x \prec y$ in~$K$. Then
\[
\conL{x',y'} =\conL{\consub{K_4}{x',y'}} = \conL{\consub{K_4}{x,y}}=\conL{x,y},
\]
verifying \eqref{I:BT5}. 

We now verify statement~\eqref{I:BT6}. The lattice $L$ is obtained by adding a relative complement $m$ of $s$ in the interval $[u,u']$ of $K_4$.
In $K_4$, $t$ and $s$ are the only upper covers of $u$ and the only lower covers of $u'$. The hypotheses of Lemma~\ref{L:bridge} thus apply to the lattice $K_4$ and its extension $L$.
So \[\conL{x_0,y_0} \leq \conL{x_1,y_1}\] if{}f either
\[
\consub{K_4}{x_0,y_0} \leq \consub{K_4}{x_1,y_1}
\]
or
\begin{align*}
\consub{K_4}{x_0,y_0} \leq &\consub{K_4}{u,t} = \consub{K_4}{a,b} \\
&\text{ and } \consub{K_4}{a',b'} =
\consub{K_4}{u', t} \leq \consub{K_4}{x_1,y_1}
\end{align*}
or
\begin{align*}
\consub{K_4}{x_0,y_0} \leq &\consub{K_4}{u',t} = \consub{K_4}{a',b'} \\
&\text{ and } \consub{K_4}{a,b} =
\consub{K_4}{u, t} \leq \consub{K_4}{x_1,y_1}.
\end{align*}
Since $K_4$ is a congruence preserving extension of $K$, we conclude that
\[\conL{x_0,y_0} \leq \conL{x_1,y_1}\] if{}f either
\[
\conK{x_0,y_0} \leq \conK{x_1, y_1}
\]
or
\[
\conK{x_0,y_0} \leq \conK{a,b} \text{ and } \conK{a',b'} \leq \conK {x_1,y_1}
\]
or
\[
\conK{x_0,y_0} \leq \conK{a',b'} \text{ and } \conK{a,b} \leq \conK {x_1,y_1},
\]
thereby verifying \eqref{I:BT6} and concluding the proof
of the theorem.
\end{proof}

Let $K$ and $L$ be as in Theorem~\ref{T:BT}, and let $\conK{a,b} \incomp \conK{a', b'}$. Then we can fuse the set $A = \set{\conK{a, b}, \conK{a' b'}}$ and get the ordered set $\Fuse(\Ji(\Con{K}), A)$. Applying Lemma~\ref{L:isom}, we get an immediate corollary:

\begin{corollary}\label{C:collapse}
Assume that $\conK{a,b} \incomp \conK{a', b'}$ and set
\[
A = \set{\conK{a, b}, \conK{a' b'}}.
\]
Then the mapping $\gf' \colon \Fuse(\Ji(\Con{K}), A) \to \Ji(\Con{L})$, with
\[
\gf' \colon \Fused \mapsto \conL{a, b} = \conL{a', b'}
\]
and \[\gf' \colon \conK{x, y} \mapsto \conL{x, y},\] otherwise, is an order isomorphism.
\end{corollary}

We now state sufficient conditions on $K$ to guarantee that attaching a bridge preserves having a minimal set of principal congruences.

\begin{theorem}\label{T:min}
Let the finite lattice $K$ be as in Theorem~\ref{T:BT} and let $K$ furthermore satisfy the following five conditions:
\begin{enumeratei}
\item \label{I:under1}If $x < i$, then $\con{x,i} = \con{b,i}$.
\item \label{I:over1} If $x > i$, then $\con{x,i} = \con{b',i}$.
\item \label{I:min} If $x < y$, then either $\con{x,y} = \one$ or $\con{x,y}$ is a join-irreducible congruence. 
\item \label{I:under2} If $x < b$ and $x \neq a$, then $\con{x,b} = \con{b,i}$.
\item \label{I:over2} If $x > b'$ and $x \neq a'$, then $\con{x,b'} = \con{b',i}$.
\end{enumeratei}
Let $L$ be obtained by attaching the bridge $\Bridge$ depicted in Figure~\ref{F:bridge} to $K$ as in Theorem~\ref{T:BT}. Then $L$ satisfies conditions~\eqref{I:under1}, \eqref{I:over1}, and \eqref{I:min}.
\end{theorem}
\begin{proof}
Let $x \in L$ with $x < i$. Then $x \in K$ and so, by \eqref{I:under1} for $K$, $\conK{x,i} = \conK{b,i}$. Then $\conL{x,i} = \conL{b,i}$, verifying \eqref{I:under1} for $L$.

\eqref{I:over1} is just the dual of \eqref{I:under1}.

Now let us denote by $o$ the zero of $K$, and so of $L$, and by $o'$ the unit of $K$, and so of $L$. By \eqref{I:under1} and \eqref{I:over1}, which were assumed for $K$ and verified for $L$, we note that if $x < i < y$, then, whether we are referring to $K$ or $L$, $\con{x,y} = \one$. Indeed,
\begin{align*}
\con{x,y} = \con{x,i} \jj \con{y,i} &= \con{b,i} \jj \con{b',i}\\
&=\con{o,i} \jj \con{o',i} = \con{o,o'} = \one.
\end{align*}

We now verify \eqref{I:min} for $L$.

Assume first that $x, y \in K$.  By \eqref{I:min} for $K$, either \[\conK{x,y} = \one_K = \conK{o,o'}\]
and so \[\conL{x,y} = \conL{o,o'} = \one_L,\] or there are $v, w \in K$ with $v \prec w$ in $K$ and with $\conK{x,y} = \conK{v,w}$.  Then $\conL{x,y} = \conL{v,w}$ and, by Theorem~\ref{T:BT}\eqref{I:BT1}, $v \prec w$ in $L$, that is, $\conL{x,y}$ is join-irreducible.

Otherwise, by duality, we can assume that $y \nin K$. If $x \prec y$ in $L$, we are done.
So we may assume that $y$ does not cover $x$ in $L$
Then $x \in K$, and so $x \leq b$.

Assume, first, that $x=b$. Then $y = u'$. So
\[
\conL{x,y} = \conL{b,t} \jj \conL{t,u'} = \conL{i,b'} \jj \conL{a',b'} =\conL{i,b'},
\]
since \[\conL{a',b'} \leq \conL{a',i} = \conL{i,b'}\] by \eqref{I:over1}. Thus if $x = b$, then $\conL{x,y}$ is join-irreducible.

Thus we are left with the case when $x < b$.

Assume further that $x=a$. Then $y$, not covering $x$, must be one of $t, m, s, u'$. Then
$\conL{u,y} = \conL{a',b'}$, and so
\[
\conL{x,y} = \conL{a,u} \jj \conL{u,y} = \conL{i,b'} \jj \conL{a',b'} = \conL{i,b'}
\]
since $\conL{a',b'} \leq \con{i,b'}$. Thus again, $\conL{x,y}$ is join-irreducible.

We are then left with the case $x < b$ and $x \neq a$.

By \eqref{I:under2} for $K$, $\conK{x,b} = \conK{b,i}$. 
Therefore, $\conL{x,b} = \conL{b,i}$. Since $y \nin K$, $y \geq t$. Then
\[
\one_L = \conL{b,b'} = \conL{b,i} \jj \conL{b',i} = \conL{x,b} \jj \conL{b,t} \leq \conL{x,y},
\]
since $b <i < b'$. So in this final case, we have $\conL{x,y} = \one_L$.

We have thus verified \eqref{I:min} for $L$, concluding the proof. 
\end{proof}

In order to repeatedly attach bridges, we state the following  easy lemma.

\begin{lemma}\label{L:attach}
Let the finite lattice $K$ be as in Theorem~\ref{T:BT}, and let $L$ be obtained by attaching the bridge $\Bridge$ depicted in Figure~\ref{F:bridge}. Assume further that there are elements $a_0, b_0 \in K$, with $a_0 \neq a$, $b_0 \neq b$, with $a_0 \prec b_0 \prec i$ in $K$, such that for any $x \in K$ with $x \neq a_0$ and $x  < b_0$, we have $\conK{x,b_0} = \conK{b_0,i}$. 
Then for any $x \in L$ with $x \neq a_0$ and $x < b_0$, 
we have $\conL{x,b_0} = \conL{b_0,i}$.
\end{lemma}
\begin{proof}
Since $x < i$, it follows that $x \in K$. 
Thus all elements considered are elements of $K$, and so
$\conL{x,b_0} = \conL{b_0,i}$, since $\conK{x,b_0} = \conK{b_0,i}$.
\end{proof}

\section{Proving Theorem~\ref{T:mainvar}(ii)}\label{S:theorem}

\subsection{Some technical results}\label{S:results}

We first summarize the properties of the lattice $\Framew P$ of G. Gr\"atzer \cite{gGb}.

\begin{lemma}\label{L:frame}
Let $P$ be a finite ordered set with a greatest element $p_0$ and let $L = \Framew P$. There is an order isomorphism $\gz_P \colon P \to \Ji(\Con{L})$ such that $\gz_P(p_0) = \one$ and such that  the following five statements hold.
\begin{enumeratei}
\item\label{I:frame1} $\bga \in \Ji(\Con{L})$ for all $\bga \in \Princ{L}$.
\item\label{I:frame2} $\conL{x,i} = \one = \gz_P(p_0)$ for all $x \in L$ distinct from $i$. 
\item\label{I:frame3} For each $p \in P$ distinct from $p_0$, there are $a_p, b_p \in L$ with $a_p \prec b_p \prec i$ such that $\conL{a_p,  b_p} = \gz_P(p)$.
\item\label{I:frame4} $a_p$ and $b_p$ are meet-irreducible for each $p \in P$ distinct from $p_0$.
\item\label{I:frame5} For each $p \in P$ distinct from $p_0$ and each $x \in P$ distinct from $a_p$, if $x < b_p$, then $\conL{x, b_p} = \one = \gz_P(p_0)$.
\end{enumeratei}
\end{lemma}

Let $P$ be a finite ordered set with exactly two maximal elements, $p_0$ and $p_1$. To prove Theorem~\ref{T:mainvar}(ii), 
we proceed by mathematical induction on the size of the 
subset $\Dg{p_0}\,\,\ii \Dg{p_1}$. First, the result that provides the base of the induction.

\begin{lemma}\label{L:ind0}
Let $P$ be a finite ordered set with exactly two maximal elements $p_0$ and $p_1$, and let $\Dg{p_0}\,\ \ii \Dg{p_1} = \es$. Then there is a finite lattice $L$ with zero $o$, unit $o'$, and element $i$ distinct from $o$ and $o'$, and there is an order isomorphism $\gz_P \colon P \to \Ji(\Con{L})$ such that the following nine statements hold.
\begin{enumeratei}
\item\label{I:one0} If $\bga \in \Princ{L}$ and $\bga \neq \one$, then $\bga \in \Ji(\Con{L})$.
\item\label{I:two0} $\conL{x, i} = \gz_P(p_0)$ for all $x \in L$ with $x < i$.
\item\label{I:three0} $\conL{x, i} = \gz_P(p_1)$ for all $x \in L$ with $ x > i$.
\item\label{I:four0} For each $p \in P$ with $p < p_0$, there are $a_p, b_p \in L$ with $a_p \prec b_p \prec i$ such that $\conL{a_p, b_p} = \gz_P(p)$.
\item\label{I:five0} For each $p \in P$ with $p < p_1$, there are $a'_p, b'_p \in L$ with $i \prec b'_p \prec a'_p$ such that $\conL{a'_p, b'_p} = \gz_P(p)$.
\item\label{I:six0} For each $p \in P-\set{p_0, p_1}$, if $p \nleq p_1$, that is, if $p < p_0$, then $a_p$ and $b_p$ are meet-irreducible in $L$.
\item\label{I:seven0} For each $p \in P-\set{p_0, p_1}$, if $p \nleq p_0$, that is, if $p < p_1$, then $a'_p$ and $b'_p$ are join-irreducible in $L$.
\item\label{I:eight0} For each $p \in P-\set{p_0,p_1}$ and each $x \in L-\set{a_p}$, if $x < b_p$, then $\conL{x, b_p} = \gz_P(p_0)$.
\item\label{I:nine0} For each $p \in P-\set{p_0,p_1}$ and each $x \in L-\set{a'_p}$, if $x > b'_p$, then $\conL{x, b'_p} = \gz_P(p_1)$.
\end{enumeratei}
\end{lemma}
\begin{proof}
We let $L$ be the base lattice for $P$, $\Base{P}$, as defined in section~\ref{S:construction}. That is, setting $P_0 = \Dg{p_0}$ and $P_1 = \Dg{p_1}$, our ordered set $P$ is then the free union of $P_0$ and $P_1$. The lattice $L$ is then $L_0 \gsum L_1$, where $L_0$ is the lattice $\Framew P_0$ and $L_1$ is the dual of the lattice $\Framew P_1$---see the details in section~\ref{S:construction}.

By Lemma~\ref{L:frame}, there are order isomorphisms $\gz_0 \colon P_0  \to \Ji(\Con{L_0})$ and $\gz_1 \colon P_1 \to \Ji(\Con{L_1})$ such that statements \eqref{I:frame1}--\eqref{I:frame5} of Lemma~\ref{L:frame} hold for $\gz_0$, $P_0$, $L_0$ and their duals (with $i', a'_p, b'_p$ replacing $i, a_p, b_p$, respectively) hold for $\gz_1$, $P_1$, $L_1$. Now, $\Ji(\Con{L})$ is the free union of the ordered sets $\Ji(\Con{L_0})$ and $\Ji(\Con{L_1})$ and so we have the order isomorphism $\gz_P \colon P \to \Ji(\Con{L})$ which is $\gz_0$ on $P_0$ and $\gz_1$ on $P_1$, and thus satisfies our statements~\eqref{I:two0}--\eqref{I:nine0}.

Only statement~\eqref{I:one0} is left to verify. So let $x, y \in L$ with $x < y$. If $x, y$ are both in $L_0$ or both in $L_1$, then $\conL{x, y} \in \Ji(\Con{L})$ by the self-dual statement Lemma~\ref{L:frame}\eqref{I:frame1}. On the other hand, if $x \in L_0$ and $y \in L_1$, then $x \leq i=i' \leq y$ and $\conL{x, i} = \conL{o,i}$ and $\conL{i', y} = \conL{i', o'}$ by statement~\eqref{I:frame2} and its dual of Lemma~\ref{L:frame}. Then $\conL{x,y} = \conL{o,o'} = \one_L$. Thus statement~\eqref{I:one0} is verified, concluding the proof.
\end{proof}

We then have:

\begin{theorem}\label{T:ind}
Let $P$ be a finite ordered set with exactly two maximal elements $p_0$ and $p_1$. Then there is a finite lattice $L$ with zero $o$, unit $o'$, and element $i$ distinct from $o$ and $o'$, and there is an order isomorphism $\gz_P \colon P \to \Ji(\Con{L})$ such that the following nine statements hold.
\begin{enumeratei}
\item\label{I:one} If $\bga \in \Princ{L}$ and $\bga \neq \one$, then $\bga \in \Ji(\Con{L})$.
\item\label{I:two} $\conL{x, i} = \gz_P(p_0)$ for all $x \in L$ with $x < i$.
\item\label{I:three} $\conL{x, i} = \gz_P(p_1)$ for all $x \in L$ with $ x > i$.
\item\label{I:four} For each $p \in P$ with $p < p_0$, there are $a_p, b_p \in L$ with $a_p \prec b_p \prec i$ such that $\conL{a_p, b_p} = \gz_P(p)$.
\item\label{I:five} For each $p \in P$ with $p < p_1$, there are $a'_p, b'_p \in L$ with $i \prec b'_p \prec a'_p$ such that $\conL{a'_p, b'_p} = \gz_P(p)$.
\item\label{I:six} For each $p \in P-(\Dg{ p_1})_P$ distinct from $p_0$, the elements $a_p$ and $b_p$ are meet-irreducible in $L$.
\item\label{I:seven} For each $p \in P-(\Dg{ p_0})_P$ distinct from $p_1$, the elements $a'_p$ and $b'_p$ are join-irreducible in $L$.
\item\label{I:eight} For each $p \in P-\set{p_0,p_1}$ and each $x \in L-\set{a_p}$, if $x < b_p$, then $\conL{x, b_p} = \gz_P(p_0)$.
\item\label{I:nine} For each $p \in P-\set{p_0,p_1}$ and each $x \in L-\set{a'_p}$, if $x > b'_p$, then $\conL{x, b'_p} = \gz_P(p_1)$.
\end{enumeratei}
\end{theorem}

\begin{proof}
We proceed by mathematical induction on the size of \[P' = (\Dg{p_0})_P \ii (\Dg{p_1})_P.\] If $P' = \es$, then Lemma~\ref{L:ind0} applies.

So let $P'$ contain at least one element. By finiteness, there is a maximal element $q \in P'$. We split $q$ into $q_0$ and $q_1$ and set $Q=\Split(P,q)$ with $q_0 <_Q p_0$ and $q_1 <_Q p_1$. Then \[(\Dg{p_0})_Q \ii (\Dg{p_1})_Q = P' - \set{q}.\] So we may assume that there is a finite lattice $K$ with zero $o$ and unit $o'$ and with an element $i$ with $o < i < o'$,  and that there is an order isomorphism  $\gz_Q \colon Q \to \Ji(\Con{K})$ such that statements \eqref{I:one}--\eqref{I:nine}, with $P$ replaced by $Q$ and $L$ replaced by $K$, hold.

By statements~\eqref{I:four} and \eqref{I:five} for $Q$ and $K$, there are $a_{q_0}, b_{q_0}, a'_{q_1}, b'_{q_1}$ in $K$ with $a_{q_0} \prec b_{q_0} \prec i \prec b'_{q_1} \prec a'_{q_1}$ and with $\conK{a_{q_0}, b_{q_0}}  = \gz_Q(q_0)$ and  $\conK{a'_{q_1}, b'_{q_1}}  = \gz_Q(q_1)$. By statements~\eqref{I:six} and \eqref{I:seven} for $Q$, $K$, $a_{q_0}, b_{q_0}$ are meet irreducible and $a'_{q_1}, b'_{q_1}$ are join-irreducible. We attach to $K$ the bridge
\[
   \Bridge(q_0, q_1) = 
   \set{a_{q_0}, b_{q_o}, i, b'_{q_1}, a'_{q_1}, t_{q_0,q_1}, 	
        u_{q_0,q_1}, u'_{q_0,q_1}, s_{q_0,q_1}, m_{q_0,q_1}}
\]
(see Figure~\ref{F:bridge2}; this construct is a generalization
of the bridge, $\Bridge(p)$, of Section~\ref{S:lattice}---in fact,
$\Bridge(p) = \Bridge(p, p)$) 
between $[a_{q_0}, b_{q_0}]$ and $[b'_{q_1}, a'_{q_1}]$, thereby getting the lattice $L$. By Theorem~\ref{T:BT}\eqref{I:BT1}, we have an isotone mapping \[\gf \colon \Ji(\Con{K}) \to \Ji(\Con{L})\] with
\[\gf \colon \conK{x,y} \mapsto \conL{x,y}\] 
for each pair $x, y \in K$ with $x \prec_K y$.

\begin{figure}
\centerline{\includegraphics[scale=.7]{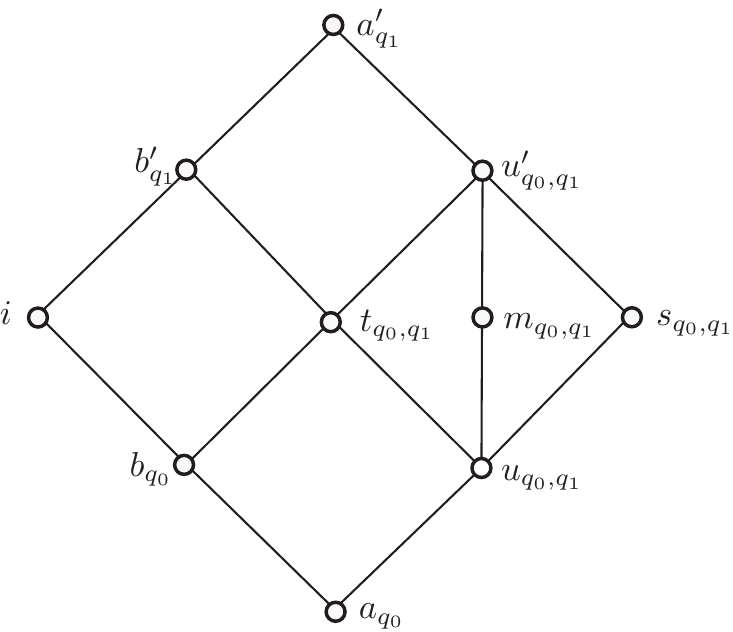}}
\caption{Notation for the bridge, $\Bridge(q_0, q_1)$}
\label{F:bridge2}
\end{figure}

Now $\conK{a_{q_0}, b_{q_0}} \incomp \conK{a'_{q_1}, b'_{q_1}}$ since $q_0 \incomp_Q q_1$ and $\gz_Q$ is an order isomorphism. Thus we fuse the two congruences; setting
\[
A = \set{\conK{a_{q_0}, b_{q_0}}, \conK{a'_{q_1}, b'_{q_1}}},
\]
we get the ordered set $\Fuse(\Ji(\Con{K}), A)$.
Now\[
\gf(\conK{a_{q_0}, b_{q_0}}) = \gf(\conK{a'_{q_1}, b'_{q_1}})
\]
since $\conL{a_{q_0}, b_{q_0}} = \conL{a'_{q_1}, b'_{q_1}}$. By Lemma~\ref{L:ump}, we get
\[
\gf' \colon \Fuse(\Ji(\Con{K}), A) \to \Ji(\Con{L}),
\]
where
\[
\gf' \colon \Fused \mapsto \conL{a_{q_0}, b_{q_0}} = \conL{a'_{q_1}, b'_{q_1}}
\]
and
\[
\gf' \colon \conK{x, y} \mapsto \conL{x, y},
\]
otherwise.

By Theorem~\ref{T:BT}\eqref{I:BT5},  $\gf$ is surjective, and so
\[
\gf' \colon \Fuse(\Ji(\Con{K}), A) \to \Ji(\Con{L})
\]
is an order isomorphism, by Theorem~\ref{T:BT}\eqref{I:BT6} and Lemma~\ref{L:isom}. The order isomorphism \[\gz_Q \colon Q \to \Ji(\Con{K})\] yields an order isomorphism 
\[
   \gz' \colon \Fuse(Q, \set{q_0,q_1}) \to \Fuse(\Ji(\Con{K}), A)
\]
 with $\gz' \colon \Fused_{\set{q_0, q_1}} \mapsto \Fused_A$ and $\gz' \colon p \mapsto \gz_Q(p)$ otherwise. We then have the order isomorphism 
\[
   \gf'\gz' \colon \Fuse(Q, \set{q_0,q_1}) \to \Ji(\Con{L})
\]  
with 
\[
   \iota_{\set{q_0,q_1}} \mapsto \conL{a_{q_0},b_{q_0}} = 
      \conL{a'_{q_1}, b'_{q_1}}
\]
and
$p \mapsto \gf\gz_Q(p)$ otherwise. By Lemma~\ref{L:isoFS}, there is an order isomorphism $P \to \Fuse(Q, \set{q_0, q_1})$ with $q \mapsto \iota_\set{q_0, q_1}$ and $p \mapsto p$ otherwise. We then get the desired order isomorphism
\[
   \gz_P \colon P \to \Ji(\Con{L})
\]
with
\begin{equation}\label{E:zPp}
   \gz_P \colon p \mapsto \gf(\gz_Q(p)) = \conL{\gz_Q(p)}
\end{equation}
for $p \neq q$
and
\[
   \gz_P \colon q \mapsto \conL{a_{q_0},b_{q_0}} = 
      \conL{a'_{q_1}, b'_{q_1}}.
\]
and, by \eqref{E:zPp} and statements~\eqref {I:four} and \eqref{I:five} for $Q$ and $K$, if $p \neq q$,
\[
   \gz_P \colon  p \mapsto
   \begin{cases}
      \gf(\conK{a_p, b_p}) = \conL{a_p, b_p}
      \text{\q\q for } p \leq p_0,\\
       \gf(\conK{a'_p, b'_p}) = 
   \conL{a'_p, b'_p} \text{\q\q for } p \leq p_1.
   \end{cases}
\]

Thus statements~\eqref{I:four} and \eqref{I:five} hold for $Q$ and $L$.

We now verify the other seven statements for $Q$ and $L$.

By statements~\eqref{I:two}, \eqref{I:three}, \eqref{I:one}, \eqref{I:eight}, and \eqref{I:nine} for $Q$,  $K$, and Theorem~\ref{T:min}, it follows that statements~\eqref{I:one}, \eqref{I:two}, and \eqref{I:three} hold for $P$, $L$.

Now let $p \in P-(\Dg{p_1})_P$ be distinct from $p_0$. Then $p \in Q-(\Dg{p_1})_Q$, and is distinct from $q_0$ as well. Then by statement~\eqref{I:six} for $Q$, $K$, the elements $a_p, {b_p}$ are meet-irreducible  in $K$. Then by Theorem~\ref{T:BT}\eqref{I:BT2}, $a_p, b_p$ are meet-irreducible in $L$, establishing statement~\eqref{I:six} for $P$, $L$. Similarly, by Theorem~\ref{T:BT}\eqref{I:BT3}, 
we get statement~\eqref{I:seven} for $P$, $L$.

Finally, statements~\eqref{I:eight} and \eqref{I:nine} for $P$, $L$ follow from the corresponding statements for  $Q$, $K$ and from Lemma~\ref{L:attach} and its dual.

By mathematical induction, the proof of the theorem is thus concluded.
\end{proof}

\subsection{An application}\label{S:application}

As an application of the results in Section~\ref{S:results},
we obtain Theorem~\ref{T:mainvar}(ii):

\begin{theorem}\label{T:mainvar2}
Let $D$ be a finite distributive lattice with exactly two dual atoms. Then $D$ has a minimal representation $L$.
\end{theorem}
\begin{proof}
By Lemma~\ref{L:corr}, $P = \Ji(D)$ has exactly two maximal elements. By Theorem~\ref{T:ind}, there is an a finite lattice $L$ satisfying statement Theorem~\ref{T:ind}\eqref{I:one} and an order isomorphism $P \to \Ji(\Con{L})$. We then have our required representation, and it is minimal by Theorem~\ref{T:ind}\eqref{I:one}. 
\end{proof}

\end{document}